\documentclass[a4paper,11pt,leqno,english]{smfart}
\usepackage{aeguill}
\usepackage{enumerate}
\usepackage{amssymb,amsmath,latexsym,amsthm}
\usepackage[T1]{fontenc}
\usepackage[hmargin=2.5cm, vmargin=3cm]{geometry}
\usepackage{url}
\usepackage[frenchb, main=english]{babel}
\usepackage[utf8]{inputenc}
\usepackage{mathrsfs}
\usepackage{xcolor}
\usepackage{comment}
\definecolor{violet}{rgb}{0.0,0.2,0.7}
\definecolor{rouge2}{rgb}{0.8,0.0,0.2}
\usepackage{tikz}
\usepackage{empheq}
\usepackage{tikz-cd}
\usetikzlibrary{matrix,arrows,decorations.pathmorphing}
\usepackage{hyperref}
\usepackage{mathpazo}
\hypersetup{
    bookmarks=true,         
    unicode=false,          
    pdftoolbar=true,        
    pdfmenubar=true,        
    pdffitwindow=false,     
    pdfstartview={FitH},    
    pdftitle={},    
    pdfauthor={},     
    colorlinks=true,       
   linkcolor=rouge2,          
    citecolor=violet,        
    filecolor=black,      
    urlcolor=cyan}           
\usepackage{enumitem}
\usepackage{appendix}

 \theoremstyle{plain}    
 \newtheorem{thm}{Theorem}[section]
\theoremstyle{plain} 
\newtheorem{bigthm}{Theorem}
\newtheorem{bigcoro}[bigthm]{Corollary}
 
 \numberwithin{equation}{section} 
 \numberwithin{figure}{section} 
 \newtheorem{cor}[thm]{Corollary} 
 \theoremstyle{plain}    
 \newtheorem{prop}[thm]{Proposition} 
 \theoremstyle{plain}    
 \newtheorem{lem}[thm]{Lemma} 
 \theoremstyle{remark}
 \theoremstyle{remark}
 \newtheorem{rem}[thm]{Remark}
 \theoremstyle{definition}
\newtheorem{exa}[thm]{Example}
\theoremstyle{plain}  

\theoremstyle{plain}
\newtheorem{setup}[thm]{Setup}
\theoremstyle{plain}

\theoremstyle{definition}
\newtheorem{defi}[thm]{Definition}

\newtheorem*{ackn}{Acknowledgements}

\newtheorem{conj}[thm]{Conjecture}

\newcommand{\C}{{\mathbb{C}}}
\newcommand{\N}{{\mathbb{N}}}
\newcommand{\PP}{{\mathbb{P}}}
\newcommand{\Q}{{\mathbb{Q}}}
\newcommand{\R}{{\mathbb{R}}}

\newcommand{\CK}{\hyperref[CK]{Condition (K)} }

\def\1{\mathbf{1}}

\newcommand{\zz}{|z|^2}
\newcommand{\zze}{|z|^2+\ep^2}

\renewcommand{\a}{\alpha}

\newcommand{\e}{\varepsilon}
\newcommand{\om}{\omega}
\newcommand{\f}{\varphi}

\newcommand{\p}{\psi}

\newcommand{\ome}{\omega_{\ep}}

\newcommand{\vp}{\varphi}

\newcommand{\ep}{\varepsilon}

\newcommand{\Ric}{\mathrm{Ric} \,}

\renewcommand{\ge}{\geq}
\renewcommand{\le}{\leq}

\newcommand{\PSH}{\operatorname{PSH}}

\title{Diameter of K\"ahler currents}
\date{\today}

\author{Vincent Guedj}

\address{Institut Universitaire de France \& Institut de Mathématiques de Toulouse; UMR 5219, Université de Toulouse; CNRS, UPS, 118 route de Narbonne, F-31062 Toulouse Cedex 9, France}

\email{vincent.guedj@math.univ-toulouse.fr}

\author{Henri Guenancia}
\address{Institut de Mathématiques de Toulouse; UMR 5219, Université de Toulouse; CNRS, UPS, 118 route de Narbonne, F-31062 Toulouse Cedex 9, France}
\email{henri.guenancia@math.cnrs.fr}

\author{Ahmed Zeriahi}

\address{Institut de Mathématiques de Toulouse; UMR 5219, Université de Toulouse; CNRS, UPS, 118 route de Narbonne, F-31062 Toulouse Cedex 9, France}

\email{ahmed.zeriahi@math.univ-toulouse.fr}

\begin{document}

\begin{abstract}  
We establish upper bounds on the diameter of compact K\"ahler manifolds endowed with
K\"ahler metrics whose volume satisfies an Orlicz integrability condition. 
Our results 
extend previous estimates due to 
Fu-Guo-Song, Y.Li, and Guo-Phong-Song-Sturm.
In particular, they do not involve any constraint on the vanishing of the volume form.
Moreover, we show that singular K\"ahler-Einstein currents have finite diameter,
provided that their local potentials are H\"older continuous.
\end{abstract} 

\maketitle

\tableofcontents

\section*{Introduction}

The classical Bonnet-Myers theorem ensures that a complete $n$-dimensional Riemannian manifold
$(M,g)$ with positive Ricci curvature ${\rm Ric}(g) \ge (n-1) K$, $K>0$, 
has finite diameter 
$$
{\rm diam}(M,g)  \le \frac{\pi}{\sqrt{K}}.
$$

 This estimate breaks down  if the Ricci curvature takes zero or negative values, 
and one then needs to add extra assumptions on the manifold.
Several works have  been devoted to establish diameter upper bounds in the last two decades
for {\it K\"ahler manifolds},
   in connection with the study of degenerate families of 
  K\"ahler-Einstein metrics (see notably \cite{Paun01,Tos09,RZ,Li20,GPTW21,GPSS22}).
  The main purpose of this article is to extend these estimates in several directions.
  
  \medskip

  We start by generalizing works of Fu-Guo-Song \cite{FGS20}
  and Guo-Song \cite{GS22},
  establishing   a  diameter upper bound under a Ricci lower bound for any metric
  whose potential is uniformly bounded, and whose
  cohomology class belongs to a compact subset of the Kähler cone.

\begin{bigthm} \label{thmA}
Let $(X,\omega_X)$ be a  compact K\"ahler manifold 
and let ${\mathcal K} \subset H^{1,1}(X,\R)$ be a compact subset of the Kähler cone of $X$.
Fix $A,B,C \in \R$ and
let $\omega$ be a K\"ahler form such that  $[\omega] \in {\mathcal K}$,
$$
\|\f_{\omega}\|_{\infty} \le C
\; \; \text{ and } \; \; 
{\rm Ric} (\omega) \ge -A \omega-B\omega_X.
$$

\noindent
Then ${\rm diam}(X,\omega) \le D$, where $D$ depends
only on $A,B,C$ and ${\mathcal K}$.
\end{bigthm} 

Besides providing a quite general statement (twisted Ricci curvature lower bound, uniform bound
on potentials instead of an Orlicz condition on the density),
our proof is completely different than previous ones. It is purely complex analytic,
using as an essential tool the semi-continuity properties of complex
singularity exponents \cite{DK} and the resolution of the openness conjecture \cite{GuanZhou15}.

Combining our techniques with the recent results of Guo-Phong-Song-Sturm \cite{GPSS22}, we can actually extend the result above to the case where $\mathcal K$ is merely a {\it bounded} subset of the Kähler cone 
such that its closure $\overline{{\mathcal K}}$ is contained in the big cone, i.e. the cone of classes represented by Kähler currents. We refer to Theorem~\ref{thm:open} for this more general version along with the precise meaning of $\varphi_\omega$.

Poincaré type metrics  have constant Ricci curvature and  infinite diameter;
we exhibit well chosen smooth K\"ahler approximants of the latter in Example \ref{Poincare}, showing
that the uniform bound on $\|\f_{\omega}\|_{\infty}$ is necessary.
On the other hand, $K$-semistable Fano manifolds which are {\it not} $K$-stable illustrate
that one cannot expect a complex proof of Myers theorem along these lines
(see Example \ref{Fano}).


 It is however difficult in general to guarantee a uniform Ricci lower bound
 on natural approximants of canonical K\"ahler currents, even when the
 Ricci curvature of the limiting object is bounded below.
 We illustrate this principle in Section \ref{sec:Riccibound} where
 we give a systematic treatment of K\"ahler metrics having one isolated singularity
 which is $U(n,\C)$-invariant (radial symmetry): we show that the Ricci curvature of
 $U(n,\C)$-invariant smooth approximants is (almost) never bounded below.
 
 \smallskip
 
Let us now go back to the recent breakthrough result of Guo-Phong-Song-Sturm \cite{GPSS22} mentioned above. It 
establishes a uniform upper bound on diameters of K\"ahler metrics $\omega$,
which only involves
\begin{itemize}
\item[$\bullet$] an upper-bound on the coholomogy class of $[\omega]$ in $H^{1,1}(X,\R)$;
\item[$\bullet$] an upper-bound on $\int_X f_{\omega} (\log f_{\omega})^p$,
where$f_{\omega}=\omega^n/dV_X$ and $p>n$;
\item[$\bullet$] a uniform lower bound $f_{\omega} \geq \gamma$, 
where $(\gamma =0)$  has small Hausdorff dimension.
\end{itemize}
We refer the reader to \cite[Theorem 1.1]{GPSS22} for a precise statement which, moreover, contains
several extra pieces of information on the Riemannian Green's function associated to $\omega$,
as well as a non-collapsing result.

\smallskip

The main result of this article is the following estimate which does not involve any
uniform lower bound on $f_{\omega}$, the latter being somehow replaced
by the positivity of the cohomology class $[\omega]$.

\begin{bigthm} \label{thmB}
Let $X$ be a compact K\"ahler manifold of complex dimension $n$ and let $\mathcal H \subset \mathcal C^\infty(X, \Omega_X^{1,1})$ be the set of Kähler forms on $X$. Let $dV_X$ be a smooth volume form and let ${\mathcal K} \subset H^{1,1}(X,\R)$ be a compact subset of the K\"ahler cone of $X$.
Given $\omega\in \mathcal H$, we set $f_{\omega}=\omega^n/dV_X$
and consider, for $A>0$ and $p>2n$ fixed,
$$
{\mathcal H}_{A,p, \mathcal K}:=\{ \omega \in \mathcal H ;  \,  [\omega]\in \mathcal K \text{ and } 
\int_X f_{\omega} |\log f_\omega|^n (\log \circ \log (f_{\omega}+3))^p  dV_X \le A \}.
$$
Then there exists a uniform constant $C=C({\mathcal K},dV_X,A,p)>0$ such that
for all $\omega \in {\mathcal H}_{A,p, \mathcal K}$,
$$
{\rm diam}(X,\omega) \le C.
$$
More precisely, given any $\gamma<p-2n$, there exists a constant $C=C({\mathcal K},dV_X,A,p,\gamma)>0$ such that for any $\omega \in {\mathcal K}_{A,p}$ and any two points $x,y\in X$, we have 
 \[d_\omega(x,y) \le \frac{C}{(\log |\log d_{\omega_X}(x,y)|)^{\frac{\gamma}{2}}}.\]
\end{bigthm}

We provide several explicit radial examples in Section \ref{sec:Riccibound}, 
showing that the assumptions made here on the density $f_{\omega}$ are close to be sharp.

\smallskip

All these results rely on fine continuity properties of Monge-Amp\`ere potentials.
If $f_{\omega}$ satisfies  \CK -introduced by Kolodziej in \cite{Kol98}-, then its 
Monge-Amp\`ere potential is uniformly bounded. 
Extending results of several authors \cite{Kol98,DDGHKZ14,GPTW21},
we show in Theorem~\ref{thm:modulus1} that one can obtain a precise control on its modulus of continuity
which is, moreover,
uniform with respect to the cohomology class $[\omega]$
(see Remark~\ref{rem:modulus}).
Theorem \ref{thmB} is thus a particular case of the following general estimate.

 \begin{bigthm}  \label{thmC}
 Let $(X,\om_X)$ be a compact Kähler manifold and let $E$ be a divisor with simple normal crossings. Let $X^\circ:=X\setminus E$ and let $T=\theta+dd^c\varphi$ be a closed positive $(1,1)$-current, where
  $\theta$ is smooth. 
  We assume that $\omega:=T|_{X^\circ}$ is a Kähler form
and that the modulus of continuity $m_\varphi$ of $\varphi$ satisfies 
$m_\varphi(r)\le \frac{C}{(\log(-\log r))^{1+\delta}}$ for some $C,\delta>0$. Then we have 
\[\mathrm{diam}(X^\circ, \omega)<+\infty.\]
\end{bigthm} 

 \smallskip
 
  In connection with the Minimal Model Program, singular K\"ahler-Einstein metrics
 $\omega_{KE}$ have been constructed in \cite{EGZ09,BBEGZ19}. These are 
 K\"ahler forms on the regular locus $V_{\rm reg}$ of 
 a K\"ahler variety $V$, whose local potentials $\f_{\rm KE}$ are  bounded
 near the singular locus $V_{\rm sing}$.
  Our method applies equally well to this singular context.
 Indeed we finally show that the finiteness of the diameter of
 these metrics follows from the conjectural  
 H\"older continuity of their Monge-Amp\`ere potentials and Theorem \ref{thmC}.\footnote{Very recently, Guo-Phong-Song-Sturm \cite{GPSS23} announced a proof of the finiteness of the diameter of singular Kähler-Einstein metrics relying on a new Sobolev inequality leading to an almost euclidean non-collapsing estimate. } 
 
 \begin{bigcoro} \label{corD}
If the K\"ahler-Einstein potentials $\f_{\rm KE}$ are H\"older continuous on $V$,
 then  \[{\rm diam}(V_{\rm reg},\omega_{\rm KE})<+\infty.\]
\end{bigcoro} 

As we explain in Section \ref{sec:Loja}, H\"older continuity is the best
regularity that makes sense intrinsically in this context.
It is known to hold on smooth manifolds \cite{Kol08,DDGHKZ14},
as well as in some singular settings \cite{HS,CS22}.

A new difficulty that occurs in the singular case is that we need some coarse control of the distance function near the singularities. This is provided by  a generalization of 
a classical $L^2$-integrability result of Demailly-Peternell-Schneider (see Lemma \ref{dist L2}).

 \smallskip
 
 \noindent {\it Method of proof and comparison with other works.}
 The starting point of the proof is the the observation made by Y.Li in \cite{Li20} 
 that the distance function $f(\cdot)=d_{\omega}(x,\cdot)$ is 
 $1$-Lipschitz, hence $0 \leq df \wedge d^c f \leq \omega$.
Assuming an $L^p$ bound on the densities $f_{\omega}$, 
Y.Li then uses H\"older regularity of the Monge-Amp\`ere potentials \cite{DDGHKZ14}
to  establish a uniform upper bound on diameters
(see \cite[Corollary 4.2]{Li20} and Proposition \ref{pro:diameterYLi}).

Guo-Phong-Song-Sturm develop in \cite{GPSS22} a systematic study of the fine properties of the
Laplace-Green function $G_{\omega}$ of $\omega$. 
Using the representation formula for $G_{\omega}$, the inequality $df \wedge d^c f \leq \omega$, and ingenious comparisons with various solutions of Monge-Amp\`ere equations associated
to $\omega$, they obtain uniform upper bounds on diameters, as well as non-collapsing results.
This requires them to impose a uniform non-vanishing condition on the densities $f_{\omega} \geq \gamma$, but the range of K\"ahler forms considered is large.

 Our proof follows an analogous path, using  
a known ambient Monge-Amp\`ere Green function rather than the detailed study of $G_{\omega}$.
A drawback is that we need fine estimates on the trace of $\omega$,
which require a control on the modulus of continuity of its Monge-Amp\`ere potentials;
the latter is only known when the reference cohomology class is K\"ahler.
An advantage is that we can directly deal with singular situations, while 
it is usually quite difficult to obtain global metric information by approximation.

  \begin{ackn} 
This work has benefited from State aid managed by the  ANR-11-LABX-0040, in connection with the research project HERMETIC, as well as by the ANR projects KARMAPOLIS and PARAPLUI and the
Institut Universitaire de France.
\end{ackn}

 \section{Modulus of continuity of Monge-Amp\`ere potentials}
 
 In the whole article we let $(X,\omega_X)$ denote a compact K\"ahler manifold of complex dimension $n$.
 We set $d=\partial+\overline{\partial}$ and $d^c=\frac{i}{2\pi}(\partial-\overline{\partial})$ 
 so that
$dd^c =\frac{i}{\pi}\partial\overline{\partial}$.

 \subsection{Quasi-plurisubharmonic functions}
 
 Recall that a function is quasi-plurisub\-harmonic if it is locally given as the sum of  a smooth and a psh function.   
 Quasi-psh functions
$\f:X \rightarrow \R \cup \{-\infty\}$ satisfying
$
\omega_X+dd^c \f \ge 0
$
in the weak sense of currents are called $\omega_X$-psh functions.

\begin{defi}
We let $\PSH(X,\omega_X)$ denote the set of all $\omega$-plurisubharmonic functions which are not identically $-\infty$.  
\end{defi}

The set $\PSH(X,\omega_X)$ is a closed subset of $L^1(X)$, 
for the $L^1$-topology. 

\smallskip

Demailly has produced various methods of regularization of quasi-psh functions.
 We recall one, together with precise estimates on the loss of positivity
 along the smoothing.
  Consider the exponential mapping with respect to the Riemannian metric induced from $\omega$. It is defined on the tangent space of a given
point $z\in X$
$$
\exp_z: T_zX\ni\zeta\mapsto \exp_z(\zeta)\in X,
$$
 by $\exp_z(\zeta)=\gamma(1)$ where $\gamma$ being the
geodesic starting from $z$ with initial velocity
$\gamma'(0)=\zeta$. Given any function $u \in \PSH(X,\omega_X)$, we define its
$\tau$-regularization  
\begin{equation}\label{phie}
\rho_\tau u(z)=\frac{1}{\tau ^{2n}}\int_{\zeta\in T_{z}X}
u(\exp_z(\zeta))\rho\Big(\frac{|\zeta|^2_{\omega_{X,z} }}{\tau ^2}\Big)\,
dV_{\omega_{X,z}}(\zeta),\ \tau>0.
\end{equation}
 Here $\rho$ is a smoothing kernel,
$|\zeta|^2_{\omega_{X,z}}$\ stands for
$\sum_{i,j=1}^ng_{i\bar{j}}(z)\zeta_i\bar{\zeta}_j$, and
$dV_{\omega_{X,z}}(\zeta)$\ is the induced measure
$ \omega_{X,z}^n/n!$.

While $\rho_{\tau}u$  already provides a quasi-psh regularization of $u$,
the loss of positivity in $dd^c \rho_{\tau}u$ is too large for applications.
Demailly combines this smoothing   with
a technique introduced by Kiselman, 
setting $U(z,w):=\rho_\tau u(z)$ for $w\in \mathbb C,\ |w|=\tau$
and considering
\begin{equation}\label{kisleg}
u_{c,\tau}(z):=\inf _{0\le t\le \tau }\Big[U(z,t)+Kt-K\tau-c
\log\Big(\frac t{\tau}\Big)\Big].
\end{equation}

\begin{lem} \cite{BD} \label{lem:regDem}
There exists $K>0$ such that the
function $U(z,t)+Kt$ is increasing in $t$ and one has the following
estimate,
\begin{equation}\label{hessian}
\omega_X+dd^c u_{c,\tau}\ge -(A  c +K\tau)\,
\omega_X,
\end{equation}
where $A \ge 0$ is such that $\Theta(T_X, \omega_X)\ge -A \om_X\otimes \mathrm{Id}_{T_X}$.
\end{lem}

 \cite[Lemma 1.12]{BD} claims a slightly finer control,
with $\tau^2$ instead of $\tau$.  
This requires to efficiently control the mixed terms
$|w| |dz\|dw|$ in \cite[Equation (1.10)]{BD}. 
These   can 
be absorbed by Cauchy-Schwarz inequality at the cost of losing one power of $\tau$,
which is how  \cite[Equation (1.10)]{BD}  allows one to
obtain \eqref{hessian}. The latter is sufficient to deal with sub-Lipschitz regularity of 
Monge-Amp\`ere potentials, as in Section \ref{CK} to follow.
A more precise control has been obtained by E.DiNezza and S.Trapani in \cite{DNT21}.

   \subsection{Diameter control by the modulus of continuity}

   \subsubsection{Dini-Campanato spaces}
  
  

 Campanato-Morrey spaces 
 provide a useful integral interpretation 
of H\"older continuity by means of uniform estimates
on mean oscillations of the function.
This approach has been extended to  more general modulus of continuity,
we shall need the following generalization due to Kovats \cite{Kov99}.

\begin{lem} \label{lem:campanato}
Let $(X,\omega_X)$ be a K\"ahler manifold of complex dimension $n$. We let 
$g$ denote the associated Riemannian metric.
Let $m:\R^+ \rightarrow \R^+$ be  an increasing subadditive continuous function such that 
$m(0)=0$ which satisfies the Dini condition
$
m_1(r)=\int_0^r \frac{m(t)}{t} dt<\infty.
$

Let $u:X \rightarrow \R$ be a measurable function. 
Assume that for each compact set $K \subset X$ there exists 
$C_K>0$ such that for all $p \in K$, for all $0 < r < \mathrm{inj}(X,g)$, 
$$
\int_{B_g(p,r)} \left| u(x)-\frac{1}{{\rm Vol }B_g(x,r)}\int_{B_g(x,r)} u \, dV_g \right|^2 dV_g(x) 
\le C_K r^{2n} m^2(r).
$$
Then $u$ is continuous and its modulus of continuity satisfies $m_u(r)=O(m_1(r))$.
\end{lem}

\subsubsection{Domination by the ambient topology}

The following local diameter control is a slight generalization of 
\cite[Corollary 4.2]{Li20}, as observed by \cite[Theorem~2]{GPTW21}.

\begin{prop} \label{pro:diameterYLi}
Let $(X,\omega_X)$ be a  K\"ahler manifold of complex dimension $n \in \N^*$.
Assume $\f \in \PSH(X,\omega_X)$ is  continuous 
in an open set $\Omega \subset X$,
 with modulus of continuity
$m_{\f}$ which satisfies the condition
$
m_1(r):=\int_0^r \frac{\sqrt{m_{\f}(t)}}{t} dt <+\infty.
$
If  $\omega:=\omega_X+dd^c \f$ is a K\"ahler form in $\Omega$,
then for each compact set $K \subset \Omega$ there exists $C_K>0$ such that
for all $p,q \in \Omega \cap K$,
$$
d_{\omega}(p,q) \le C_K \, m_1 \circ d_{\omega_X}(p,q).
$$
\end{prop}

We let here $d_{\omega}$ denote the 
Riemannian distance associated to the Kähler form $\omega$.

\begin{proof}
Fix $p \in K$ and let $\rho: x \in \Omega \mapsto d_{\omega}(x,p) \in \R^+$ denote the distance
function with respect to the K\"ahler form $\omega=\omega_X+dd^c \f$.
The function $\rho$ is $1$-Lipschitz with respect to $d_{\omega}$ hence
$\nabla_{\omega} \rho$ is well defined almost everywhere with
$|\nabla_{\omega} \rho| \le 1$. We infer
$$
|\nabla_{\omega_X} \rho|^2 \le {\rm Tr}_{\omega_X}(\omega) |\nabla_{\omega} \rho|^2 \le n \frac{\omega \wedge \omega_X^{n-1}}{\omega_X^n}.
$$

Let $0 \le \chi \le 1$ be a cut off function such that 
$\chi \equiv 1$ in a neighborhood of $B_r:=B_{\omega_X}(p,r)$, ${\rm Supp}(\chi) \subset B_{2r}:=B_{\omega_X}(p,2r)$
and $dd^c \chi \le Cr^{-2} \omega_X$.
Note that $C=C_K$ can be chosen uniformly with respect to $p \in K$.
 It follows from Chern-Levine-Nirenberg inequalities that
$$
\int_{B_r} |\nabla_{\omega_X} \rho|^2 \omega_X^n 
\le  n \int_X \chi \omega \wedge \omega_X^{n-1} 
 \le  C r^{2n} +n \int_X [\f -\inf_{B_{2r}} \f ] dd^c \chi \wedge \omega_X^{n-1} 
 \le  C' m_{\f}(2r) r^{2n-2}.
$$
Since $m_{\f}(2r) \le 2 m_{\f}(r)$, it follows from Poincar\'e inequality that
 $$
 \int_{B_r} \left| \rho(x)-\frac{1}{{\rm Vol}(B_r(x))}\int_{B_r(x)} \rho \, \omega_X^n \right|^2 \omega_X(x)^n \le C_P r^2 \int_{B_r} |\nabla_{\omega_X} \rho|^2 \omega_X^n 
 \le C''  m_{\f}(r) r^{2n}.
 $$
 
 Of course $X$ may not satisfy a global Poincaré inequality but it   holds in a neighborhood of $K$. Thus $\rho$ belongs to a  generalized Morrey-Campanato space and Lemma \ref{lem:campanato}
applied with  $m=\sqrt{m_{\f}}$ ensures that for all $p,q \in K$,
$
d_{\omega}(p,q) \le  C m_1 \circ d_{\omega_X}(p,q).
$
\end{proof}

   \subsection{Condition (K)}
   \label{CK}
   
   In honor of the breakthrough result of S.Ko\l odziej \cite{Kol98}, we introduce the following notion.    
   \begin{defi}
   A positive Radon measure $\mu$ on $X$ satisfies \CK if there exists 
    $dV_X$  a  smooth volume form, $f \ge 0$  a Lebesgue-measurable function,
$w: \R^+ \rightarrow  [1,+\infty)$ 
 a convex increasing weight 
 s.t. $\mu=f dV_X$ with $\int_X w \circ f \, dV_X <+\infty$,
 where
 $$
 w(t)= t (\log t)^n (h \circ \log \circ \log (t+3))^n 
 \; \; \text{ and } \; \;
 \int^{+\infty} \frac{dt}{h(t)}<+\infty.
 $$
   \end{defi}

 Let $\omega_X$ be a K\"ahler form and let $\mu$ be a positive Radon measure 
 normalized so that $\mu(X)={\rm Vol}_{\omega_X}(X)$.
  If $\mu$ satisfies \CK, it has been proved by Kolodziej (see \cite[Theorem 2.5.2]{Kol98}) that 
  there exists a unique continuous $\omega$-psh function such that
  $$
  (\omega_X+dd^c \f)^n=\mu=fdV_X,
  $$
  up to an additive constant, together with a uniform a priori bound on ${\rm Osc}_X(\f)$.

  \smallskip

   We now establish   a precise  control of the modulus of continuity
  of the solution $\f$, when the density
  $f$ satisfies various integrability conditions.

\begin{thm} \label{thm:modulus1}
Let $(X,\omega_X)$ be a compact K\"ahler manifold of complex dimension $n \in \N^*$.
Let $\mu=f dV_X$ be a positive Radon measure which satisfies \CK
with weight $w$.
Let $\f \in \PSH(X,\omega_X) \cap {\mathcal C}^0(X)$ be a solution of the Monge-Amp\`ere equation
$$
(\omega_X+dd^c \f)^n =f dV_X.
$$
We let $m_{\f}$ denote the modulus of continuity of $\f$ and fix $\e>0$.

\smallskip

1) If $w(t) = t^{1+\e}$ then $m_{\f}(r) \le C_{\a} r^{\a}$,
for all $0< \a < \frac{2\e}{n(1+\e)+\e}$.

\medskip

2) If $w(t) = t (\log t)^{n+\e}$, then 
$m_{\f}(r) \le C (-\log r)^{-\e/n}$.

\medskip

3) If $w(t) = t (\log t)^n (\log \circ \log (t+3) )^{n+\e}$, then 
$m_{\f}(r) \le C (\log (-\log r))^{-\e/n}$.
\end{thm}

The first item 
has been established by Kolodziej in \cite{Kol08} (see also \cite[Theorem A]{DDGHKZ14}).
The condition $\int_0 r^{-1} \sqrt{m_{\f}(r)} dr<+\infty$  is always satisfied in this case.

The second item has been addressed by Guo-Phong-Tong-Wang in \cite{GPTW21},
by using a new method introduced in \cite{GPT23}. 
 The condition $\int_0 r^{-1} \sqrt{m_{\f}(r)} dr<+\infty$ is satisfied iff $\e>2n$, while
it   is never satisfied in the third case.

 \begin{proof}
 We provide a unified proof of these results,
 following the approach and notations of \cite[Theorem D]{DDGHKZ14}.
 Since 1) is treated there, we first treat 2).
For $w(t) = t (\log t)^{n+\e}$, it follows from H\"older-Young inequality 
that
$$
\mu=f dV_X \le C_0 \, {\rm Cap}_{\omega_X}^{1+\e/n}.
$$
  \cite[Proposition 2.6]{EGZ09} shows that if $\p$ is bounded and
 $\omega_X$-psh, then for all $0 < \delta <1$,
 \begin{equation} \label{eq:stab2}
  \|(\p-\f)_+\|_{L^{\infty}} \le \delta+C_0 \left[ {\rm Cap}_{\omega_X}\left( \f-\p <-\delta \right)  \right]^{\e/n^2},
 \end{equation}
where $(\p-\f)_+:=\max(\p-\f,0)$ denotes the positive part of $\p-\f$.
 Now
 \begin{eqnarray*}
  {\rm Cap}_{\omega_X}\left( \f-\p <-\delta \right) 
  &\le & C_1 \delta^{-n}  \| \1_{\{ \p-\f >\delta \}} \|_{L^{\chi^*}} \\
& \le & \frac{C_2}{\delta^n [-\log {\rm Vol}(\p-\f>\delta)]^{n+\e}} \\
&\le & \frac{C_3}{\delta^n \left[-\log \frac{\|(\p-\f)_+\|_{L^1}}{\delta} \right]^{n+\e}},
 \end{eqnarray*}
 where the latter inequality follows from Chebyshev inequality.
 
 \smallskip
 
We set $\Phi(z,w):=\rho_\tau \f(z)$ for $w\in \mathbb C,\ |w|=\tau$
and consider
$$
\f_{c,\tau}(z):=\inf _{0\le t\le \tau }\Big[\Phi(z,t)+Kt-K\tau-c
\log\Big(\frac t{\tau}\Big)\Big].
$$
Recall from Lemma \ref{lem:regDem} that
$
\omega_X+dd^c \f_{c,\tau}\ge -(A  c +K\tau)\,
\omega_X,
$
thus the function
$$
\psi_{c,\tau} :=   \left(1 - [A c+ K \tau] \right) \f_{c,\tau} 
$$
is $\omega_X$-plurisubharmonic on $X$. 
In what follows we set 
$
\psi_\tau := \psi_{c(\tau),\tau},
$
where $c=c(\tau)$ is chosen below so that $c(\tau)=o(1)$ and $\tau=o(c(\tau))$.

Observe  that $\f_{c,\tau} \le \rho_t \f$.
Since $\f$ is bounded we can shift by an additive constant and assume  that $\f \ge 1$.
For $\tau$ small enough this ensures that $\f_{c,\tau} \ge 0$,
hence
$$
\psi_{\tau}= \f_{c,\tau} - [A \tau+ K \tau] \f_{c,\tau}
 \le \f_{c,\tau}  \le \rho_{\tau} \f.
$$

Recall that  $\|\f-\rho_{\tau}\f\|_{L^1}=O(\tau^2)$.
Since
$
\|(\psi_{\tau}-\f)_+\|_{L^1} \le \|\rho_{\tau}\f -\f\|_{L^1}
$
we infer
 $$
 \| (\psi_{\tau}-\f)_+\|_{L^{\infty}}  \le \delta+\frac{C_4}{\delta^{\frac{\e}{n}} \left[-\log \tau^2 \delta^{-1} \right]^{\frac{\e(\e+n)}{n^2}}}.
 $$
 We optimize the choices of $\tau, \delta$ by choosing
 $
 \tau^2=\delta\exp \left( - C_4\delta^{-n/\e} \right),
 $
 which leads to 
 $$
  \|(\psi_{\tau}-\f)_+\|_{L^{\infty}} \le \frac{C_5}{(-\log \tau)^{\e/n}}.
 $$

We choose $c(\tau)=(-\log \tau)^{-\e/n}$.
  The previous inequality yields
 $$
  \p_{\tau}  \le \f+\frac{C_6}{(-\log \tau)^{\e/n}}
  =\f+C_6 c(\tau).
 $$

 This provides  a uniform lower bound on $t=t(z)$ which realizes the infimum in the definition
 of $\f_{\tau,\tau}$. Recall indeed that $\rho_t \f+Kt \ge \f$, hence
 \begin{eqnarray*}
 \f(z)+C_6 c(\tau) &\ge &  \rho_t\f(z)+Kt-K\tau- c(\tau) \log\Big(\frac t{\tau}\Big)\\
  &\ge& \f(z)-K \tau -c(\tau) \log\Big(\frac t{\tau}\Big)
  \end{eqnarray*}
which leads to
 $
 \log\Big(\frac t{\tau}\Big)  \ge \frac{-K \tau}{c(\tau)}-C_6 \ge -C_7.
 $
 Thus $\exp(-C_7) \tau \le t \le \tau$, showing that
  $\rho_{\tau}\f \le\f+ C_8 c(\tau)$.
    Using that $\f$ is quasi-plurisubharmonic, we
 finally observe that $\f \le \rho_{\tau} \f+B\tau^2 \le \f+B'c(\tau)$,
 since . Thus
for all $0<\tau<\tau_0$,
 $$
  \|\f-\rho_{\tau}\f \|_{L^{\infty}} \le \frac{C_9}{(-\log \tau)^{\e/n}}.
 $$
 This provides the desired control on the modulus of continuity of $\f$, as 
 convolutions $\rho_t \f(z)$ and $\sup_{B(z,t)} \f$ are uniformly comparable
 for quasi-psh functions (see \cite{Z20}).

 \medskip
 
 We finally take care of the third item.
 The proof follows the same path, but
 the stability estimate \eqref{eq:stab2} has to be modified as follows :
 \begin{equation} \label{eq:stab3}
\max_X (\psi-\varphi) \le  \frac{B'}{\left[\log \left(- \log (\Vert \psi - \varphi\Vert_1)\right)\right]^{\varepsilon \slash n}},
\end{equation}
provided that $\Vert \psi-\varphi\Vert_1 < 1 \slash e$, where $B' > 0$ is a uniform constant.

   Indeed for $w (t) = t  (\log  t)^n (\log \circ \log (t+3) )^{n + \varepsilon}$,    
   H\"older-Young inequality  allows one to show
   that there exists $A > 0$ such that for any Borel subset $K \subset X$, 
\begin{equation} \label{eq:capestimate}
\mu (K) \, \le \,  A   \Vert f\Vert_w   \,  
\text{Cap}_{\omega_X} (K) \left(- \log  \text{Cap}_{\omega_X} (K)\right)^{-1 - \varepsilon \slash n}.
\end{equation}
Arguing  as in \cite[Proposition 2.6]{EGZ09} we 
infer that for any bounded $\omega$-plurisubharmonic functions  $\psi$ such that $\psi \ge  \varphi$ 
and any $\delta \in ]0,1[$,
$$
\|\psi-\varphi\|_{L^{\infty}} \le \delta +  \frac{B}{\left[ (- \log  \left(\text{Cap}_{\omega_X}\{ \psi-\varphi > \delta\}\right)
\right]^{\varepsilon\slash n}},
$$
The comparison principle  yields
$
\text{Cap}_{\omega_X} (\{ \psi-\varphi > \delta\}) \le C \delta^{- n}  \int_{\{\psi - \varphi > \delta\}} (\omega_X + dd^c \varphi)^n .
$
Since $(\omega_X + dd^c \varphi)^n = f d V$, it follows from  H\"older-Young inequality  that 
\begin{eqnarray*}
\text{Cap}_{\omega_X} (\{ \psi-\varphi > \delta\}) &\le &  C \delta^{- n} \Vert f\Vert_w \Vert  {\bf 1}_{\{\psi - \varphi > \delta\}}\Vert_{w^*} \\
& \le & \frac{C' \Vert f\Vert_w }{  \delta^{n} \left[- \log (-\log \text{Vol} (\{\psi - \varphi > \delta\})\right]^{\varepsilon \slash n}} \\
& \le & \frac{C' \Vert f\Vert_w }{  \delta^{n} \left[- \log (-\log  (\delta \slash \Vert \psi - \varphi\Vert_1)\right]^{\varepsilon \slash n}} .
\end{eqnarray*}
We infer
$$
\max_X (\psi-\varphi) \le \delta +  {\frac{B}{\left[-\log \left(- \log (\delta \slash \Vert \psi-\varphi\Vert_1)\right)\right]^{\varepsilon\slash n}}},
$$
provided that $\Vert \psi-\varphi\Vert_1 < 1 \slash e$. Optimizing the right hand side by taking
$$
\delta = \frac{1}{\left[\log \left(- \log (\Vert \psi - \varphi\Vert_1)\right)\right]^{\varepsilon \slash n}},
$$
 we obtain \eqref{eq:stab3}. Using Demailly-Kiselman regularization technique as above, this eventually yields the
 predicted control on the modulus of continuity.
 \end{proof}

 \begin{rem} \label{rem:modulus}
 Observe that   the constants C only weakly depend on the reference K\"ahler form $\omega_X$.
 In particular if there exist K\"ahler forms $\omega_1,\omega_2$
 such that $\omega_1 \le \omega_X \le \omega_2$, then these constants can be shown to only
 depend on the forms $\omega_1,\omega_2$. This will be useful in proving Theorem \ref{thmB}
 in Section \ref{sec:main}.
 \end{rem}

\section{Diameter control using a Ricci lower bound}

In this section we establish diameter upper bounds  by
requiring a Ricci lower bound.

     \subsection{Good representative of positive cohomology classes}
     
 Let $(X,\omega_X)$ be a compact Kähler manifold and let $dV_X:=\omega_X^n$.  
  Pick ${\mathcal B}=\{\theta_1,\ldots,\theta_s\}$ smooth closed differential forms
  whose cohomology classes constitue a basis of the finite dimensional vector space $H^{1,1}(X,\R)$.
  If $\omega$ is a K\"ahler form, it follows from the $\partial \overline{\partial}$-lemma that 
  there exists unique $(t_1,\ldots,t_s) \in \R^s$ and  
  $\f^{\mathcal B}_{\omega} \in \PSH(X,\sum_j t_j \theta_j)$ such that
  $$
  \omega=\sum_{j=1}^s t_j \theta_j+dd^c \f^{\mathcal B}_{\omega},
  \; \; \text{ with } \; \; \int_X \f^{\mathcal B}_{\omega} dV_X=0.
  $$
  
  The following elementary result is left to the reader.
  
  \begin{lem} \label{lem:boundedclasses}
  A subset ${\mathcal K}$ of the K\"ahler cone is bounded in $H^{1,1}(X,\R)$
  if and only if there exists $C>0$ such that for all $\omega \in {\mathcal K}$,
  one has $[\omega] \cdot [\omega_X]^{n-1}=\int_X \omega \wedge \omega_X^{n-1} \leq C$.
  
  Fix ${\mathcal B}$ and ${\mathcal B}'$ two bases of differential forms,
  and ${\mathcal K}$ a bounded subset of $H^{1,1}(X,\R)$.
  The following properties are equivalent.
  \begin{itemize}
  \item there  exists $C_{{\mathcal B}, {\mathcal K}}>0$ such that 
  $\|\f^{\mathcal B}_{\omega}\|_{\infty} \leq C_{{\mathcal B}, {\mathcal K}}$
  for all $\omega \in {\mathcal K}$;
  \item there  exists $C_{{\mathcal B}', {\mathcal K}}>0$ such that 
  $\|\f^{\mathcal B'}_{\omega}\|_{\infty} \leq C_{{\mathcal B}', {\mathcal K}}$
  for all $\omega \in {\mathcal K}$.
  \end{itemize}
    \end{lem}
    
    To lighten notations in what follows we get rid of the upper-script and simply write $\f_{\omega}$
    to denote the normalized potential of $\omega$ with respect to a given
    basis ${\mathcal B}$.
 
 \smallskip
   
   We now explain, given particular bounded subsets of the Kähler cone, how to 
   choose a basis ${\mathcal B}$ with specific curvature properties.

     \begin{lem} \label{lem:goodref}
 Let ${\mathcal K} \subset H^{1,1}(X,\R)$ be a bounded subset of the K\"ahler cone of $X$.
 \begin{enumerate}
 \item If ${\mathcal K}$ is closed, then
 there exists $C=C(\mathcal K)>0$ such that for each $\alpha \in {\mathcal K}$, 
 one can find a K\"ahler form $\omega_{\alpha} \in \alpha$ such that 
 $C^{-1} \omega_X \le \omega_{\alpha} \le C \omega_X$
 \item If $\overline{{\mathcal K}} \subset H_{\rm big}^{1,1}(X,\R)$, then 
 there exists $C=C(\mathcal K)>0$ 
 and finitely many quasi-psh functions 
 with analytic singularities $\rho_i$
 such that for each $\alpha \in {\mathcal K}$, 
 one can find a smooth closed $(1,1)$-form $\eta_{\alpha} \in \alpha$  such that 
 $C^{-1} \omega_X \le \eta_{\alpha}+dd^c \rho_{i_{\alpha}}$ and $\eta_{\alpha} \le C \omega_X$.
 \end{enumerate}
  \end{lem}

 \begin{proof}
 Let $\theta_1,\ldots,\theta_s$ be smooth closed $(1,1)$-forms such that the
 cohomology classes $\{\theta_j\}$ generate the finite dimensional vector space $H^{1,1}(X,\R)$. We endow the latter with the euclidean structure induced from the choice of the above basis. 
 
 Assume first that $\mathcal K$ is closed.
  Given $\alpha\in \mathcal K$, let us pick $\omega\in \alpha$ an arbitrary Kähler form. For $t\in \mathbb R^s$, consider $\omega_t:=\omega+\sum_{k=1}^s t_k \theta_k$. Clearly, there exist $\ep_\omega, C_\omega>0$ such that for $|t|<\ep_\omega$, we have $C_\omega^{-1} \omega_X \le \omega_t \le C_\omega \omega_X$. 
 By compactness of $\mathcal K$, one can find a finite set $I$ and Kähler forms $(\omega_i)_{i\in I}$ along with positive numbers $\ep_i, C_i>0$ such that
 $\bigcup_{i\in I} B([\omega_i], \ep_i) \supset \mathcal K$ and,
 for $|t|<\ep_i$, the form $\omega_{t,i}:=\omega_i+\sum_{k=1}^s t_k \theta_k$ satisfies 
 \[C_i^{-1} \omega_X \le \omega_{t,i} \le C_i \omega_X.
 \]
  The first item follows with the choice of the constant $C:=\max_{i \in I} C_i$. 
  
  \smallskip
  
 We now treat the case when $\overline{\mathcal K}$ may leave the K\"ahler cone, but still
 consists of big classes (which are also nef by definition).
Given $\alpha\in \overline{\mathcal K}$, we let $\eta+dd^c \rho \in \alpha$ be an arbitrary Kähler current with analytic singularities. 
By definition $\rho$ is smooth in some Zariski open set, and $\delta \omega_X \leq \eta+dd^c \rho $
for some $\delta>0$.
Consider $\eta_t:=\eta+\sum_{k=1}^s t_k \theta_k$
for $t\in \mathbb R^s$. 
 Clearly, there exist $\ep, C>0$ such that for $|t|<\ep$, we have 
 $$
 C^{-1} \omega_X \le \eta_t +dd^c \rho
 \; \; \text{ and } \; \; \eta_t \le C \omega_X.
 $$
  The second  item follows now by compactness of $\overline{{\mathcal K}}$.
  \end{proof}

   \subsection{Exponential integrability of Ricci deviations}

   We now provide a proof of a Theorem~\ref{thmA} relying on purely complex-analytic methods. Actually, the theorem below allows for the Kähler class to degenerate (unlike in Theorem~\ref{thmA}) and we will appeal to \cite{GPSS22} in a crucial way to be able to treat the degenerate setting.
   
   \begin{thm} \label{thm:open}
Let $(X,\om_X)$ be a  compact K\"ahler manifold and let ${\mathcal K} \subset H^{1,1}(X,\R)$ be a bounded subset of the K\"ahler cone of $X$
such that $\overline{{\mathcal K}} \subset H_{\rm big}^{1,1}(X,\R)$.
Fix $A,B,C \in \R$ and
let $\omega$ be a K\"ahler form such that  $[\omega] \in {\mathcal K}$,
$$
\|\f_{\omega}\|_{\infty} \leq C
\; \; \text{ and } \; \; 
{\rm Ric} (\omega) \ge -A \omega-B\omega_X.
$$

\noindent
Then ${\rm diam}(X,\omega) \le D$, where $D$ depends
on $A,B,C$ and ${\mathcal K}$.
\end{thm}

Upper bounds on the diameter had been previously obtained under Ricci lower bound hypotheses
by P\u{a}un \cite[Theorem 2]{Paun01}, 
Tosatti \cite[Theorem 3.1]{Tos09},
 Zhang \cite[Theorem 2.1]{RZ} and Fu-Guo-Song \cite{FGS20} (see Corollary \ref{cor:fgs}).
 
   \begin{rem} Let us formulate a couple of comments on the result above. 
   \begin{enumerate}[label=$(\roman*)$]
   \item The assumptions made on ${\mathcal K}$ are equivalent to the existence of $C>0$
   such that for all $\a \in {\mathcal K}$ one has
   $C^{-1} \leq \a^n$ and $\alpha \cdot [\omega_X]^{n-1} \leq C$.
   \item The main outcome  is that under the assumptions, one can write $\omega^n= f dV_X$, with $\|f\|_p$ uniformly bounded, $p>1$, i.e. \CK is satisfied in a strong way. 
   
\item The a priori bound on $\|\varphi_\omega\|_\infty$ does not follow from the control on the Ricci curvature from below, as Example~\ref{Fano} below shows. 
 Moreover, it can happen that the diameter is unbounded in the absence of the a priori 
 bound of $\varphi_{\omega}$, cf Example~\ref{Poincare}.
\end{enumerate}
\end{rem}

   \begin{proof}
    We can assume without loss of generality
    that the reference K\"ahler form 
    $\omega_X$ is scaled so that   $\int_X \omega_X^n=1$,
    and that $A \geq 0$.
     Let $\alpha:=[\omega]$; we write $\omega=\eta_\a+dd^c \varphi_\omega$, where
     $\eta_{\a}$ is provided by Lemma \ref{lem:goodref}.
         From now on, we simply write $\varphi=\varphi_\omega$. 
  
     By assumption ${\rm Ric} (\omega)= -A\omega -B\omega_X+T$,
  where $T \ge  0$ is a semi-positive closed $(1,1)$-form. It follows from
  the $\partial\overline{\partial}$-lemma that 
  $$
  T={\rm Ric} (\omega_X)+B\omega_X+A\eta_\a+dd^c \p,
  $$
   for some quasi-psh function $\p$ which is uniquely determined up to an additive constant.
   We fix the latter by  imposing
   $
   \omega^n =(\eta_\alpha+dd^c \f)^n=e^{A\f-\p} \omega_X^n.
   $
We easily check that 
  \begin{equation}
  \label{borne psi}
  \int_X e^{-\psi} \omega_X^n \le e^{A C} [\omega]^n, \quad \mbox{and} \quad \sup_X \psi \ge-AC-\log [\omega]^n 
  \end{equation}
The assumptions made on ${\mathcal K}$ now ensure
that we can find a constant $C_{\mathcal K}>0$ such that  one has a uniform upper bound $\int_X e^{-\psi} \omega_X^n \le C_{\mathcal K}e^{A C} $, as well as a uniform lower bound $\sup_X \psi \geq -AC-C_{\mathcal K}$.

   \medskip
   
   We set $M:=\sup_X \psi$. We proceed in two steps. In the first step, we show that there exists a constant $D=D(A, B, C, \mathcal K, M)$ such that $\mathrm{diam}(X, \omega) \le D$. In the second step, we show that one can bound $M$ in terms of $A, B,C$ and $\mathcal K$. 
   
      \medskip
   
\noindent {\bf Step 1.} {\em Diameter bound relying on an upper bound for $\psi$.}
Given any smooth form $\chi$ and any real number $E>0$, the set 
   $$
   {\mathcal C}_{\chi,M,E}=\{ u \in \PSH(X,\chi), \, \sup_X u \le  M
   \text{ and } \int_X e^{-u}\omega_X^n \le  E \}
   $$
is a compact subset of $L^1(X)$, as follows from Fatou's and Hartogs lemmas  (see \cite{GZbook}).   
We use this observation here with  (cf Lemma \ref{lem:goodref})
$$
\chi:={\rm Ric} (\omega_X)+(B+AC_{\mathcal K}')\omega_X \geq {\rm Ric} (\omega_X)+B\omega_X+A\eta_\a
$$
and $E:=e^{AC}C_{\mathcal K}$, and simply write $\mathcal C:={\mathcal C}_{\chi,M,E}$. Thanks to \eqref{borne psi}, we have $\psi \in \mathcal C$. 

Any function $u$ in ${\mathcal C}$ satisfies $\int_X e^{-u} \omega_X^n <+\infty$,
hence its complex singularity exponent is larger than one $c(u) \ge  1$.
Since quasi-psh functions are integrable at the critical exponent $c(u)$
\cite{BoB15, GuanZhou15}, and since $u \mapsto c(u)$ is lower semi-continuous
\cite[Theorem 0.2.1]{DK}, we can find $p>1$ such
that $p \le  c(u)$ for all $u \in {\mathcal C}$.

It follows now from \cite[Theorem 0.2.2]{DK} that
$$
\sup_{u \in {\mathcal C}} \int_X e^{-p u} \omega_X^n \le  B,
$$
for some uniform constant $B>0$. In particular 
$e^{-\p} \in L^p$ and $\|e^{-\p}\|_{L^p(\omega_X^n)}$ is uniformly bounded.
We infer that $\omega^n=f \omega_X^n$ with $f \in L^p$
and $\|f\|_{L^p(\omega_X^n)} \le  B'$ uniformly bounded.

When ${\mathcal K}$ is compact it  follows from Theorem \ref{thm:modulus1} that $\f$ is H\"older continuous, with a uniform control on its modulus of continuity as $[\omega]$ moves
inside ${\mathcal K}$ (see Remark~\ref{rem:modulus}).
We conclude with Proposition~\ref{pro:diameterYLi} that 
${\rm diam}(X,\omega)$ is uniformly bounded from above.
When  $\overline{{\mathcal K}}$ is merely a compact subset
of the big cone, one can invoke \cite[Theorem~1.1]{GPSS22} to conclude,
noting that $f=e^{A\f-\p} \geq c>0$ is uniformly bounded below.

   \medskip
   
\noindent {\bf Step 2.}   {\em Upper bound for $\psi$ depending only on $A,B, C$ and $\mathcal K$.}
We fix $C_2>0$ such  that the holomorphic bisectional  curvature 
${\rm Bisec} \, \omega_X \le  C_2$ is bounded from above. 
Recall that $\omega=\eta_{\a}+dd^c \f$ and Lemma \ref{lem:goodref} provides us with
a function $\rho$ with analytic singularities and  $\delta=\delta({\mathcal K})>0$
such that $\eta_{\alpha}+dd^c \rho \geq \delta \omega_X$.

It follows from Chern-Lu inequality 
(see e.g. \cite[Proposition 7.2]{Rub14})
that
$$
\Delta_{\omega} \log {\rm Tr}_{\omega}(\omega_X) \ge  -A-(2B+C_2) {\rm Tr}_{\omega}(\omega_X).
$$
Setting $H=\log {\rm Tr}_{\omega}(\omega_X)-C_3 (\f-\rho)$ we infer
\begin{eqnarray*}
\Delta_{\omega} H &\ge &  -A-(2B+C_2){\rm Tr}_{\omega}(\omega_X) -C_3 \Delta_{\omega}(\f-\rho) \\
&=& -A-C_3n-(2B+C_2){\rm Tr}_{\omega}(\omega_X) +C_3 {\rm Tr}_{\omega}(\eta_{\a}+dd^c \rho) \\
&\geq& -A-C_3n+{\rm Tr}_{\omega}(\omega_X),
\end{eqnarray*}
by choosing $C_3 \delta=1+2B+C_2$.

Since $\rho$ is smooth in some Zariski open set $\Omega$ and $\rho \rightarrow -\infty$ on $\partial \Omega$, we can apply the maximum principle in $\Omega$.
Together with the uniform bound on $\f$, we obtain
$$
{\rm Tr}_{\omega}(\omega_X) \le  C_4e^{-C_3 \rho},
\; \text{ hence } \; 
\omega_X \le  C_4 e^{-C_3 \rho} \omega.
$$
We infer $\omega_X^n \le  C_4^n e^{-nC_3 \rho} \omega^n=e^{A\f-\p-nC_3\rho+C_5}   \omega_X^n$,
hence $nC_3 \rho+\p \leq C_5+A\varphi\le C_6$.
Since $\int_X \rho \omega_X^n$ is under control, this yields a uniform upper bound on
$\int_X \p \omega_X^n$, hence on $\sup_X \p$
(see \cite[Proposition 8.5]{GZbook}).
   \end{proof}

   \subsection{Examples and consequences}

   \subsubsection{Infinite diameter or oscillation}
   
   This first example 
   exhibits a family of K\"ahler metrics with unbounded diameters, yet
   whose Ricci curvature is  bounded below.

   \begin{exa}
\label{Poincare}
Let $X$ be a projective manifold on which there exists a smooth ample divisor $D$ such that $K_X+D$ is ample  (e.g. $X=\mathbb P^n$ and $D$ is a smooth hypersurface of degree $d> n+1$). 
Pick a smooth hermitian metric $h$ with positive curvature on $D$ and a volume form $dV_X$ 
such that $\om_X:=-\mathrm{Ric}(dV_X)+i\Theta_h(D)\ge 0$. Next, let $s$ be a section of $D$ such that $(s=0)=D$. Set $\psi_\ep:=\log(|s|_h^2+\ep)$; one can check that $i\Theta_h(D)+dd^c \psi_\ep \ge 0$. We consider the solution $\varphi_\ep \in \PSH(X, \om_X)$ of the Monge-Ampère equation
   \[(\omega_X+dd^c\vp_\ep)^n=e^{\vp_e-\psi_\ep}dV_X,\]
and we set $\omega_\ep=\omega_X+dd^c \vp_\ep$. Clearly, one has $\Ric \ome \ge -\ome$ but we claim that $\mathrm{diam}(X,\ome)\to +\infty$. Indeed, it is somehow classical 
(see e.g. the proof of \cite[Theorem 4.5]{DnL15})
to see that $\omega_\ep$ converges locally smoothly on $X\setminus D$ to the Kähler-Einstein metric $\om_0$ with Poincaré singularities along $D$ constructed by R. Kobayashi \cite{KobR} and Tian-Yau \cite{TY87}. The latter metric is complete, hence has infinite diameter. 
Thus the diameter of $(X, \om_\ep)$ cannot be bounded. 
 Of course $\mathrm{Osc}_X \vp_\ep \to +\infty$. 
\end{exa}

   We now  provide a family of K\"ahler metrics with uniformly positive Ricci curvature and bounded diameters, 
   whose potentials are not uniformly bounded.
   
      \begin{exa}  \label{Fano}
Let $X$ be a $K$-semistable Fano manifold which is not $K$-stable, hence does not admit a 
K\"ahler-Einstein metric. A classic example is a general deformation of the Mukai-Umemura threefold \cite{Tian97}, cf also \cite{ChiLi17} for the general picture. 	

Pick $\omega_X \in c_1(X)$ a K\"ahler form, and let $h$ denote the
unique smooth function such that $\Ric \omega_X=\omega_X+dd^c h$
and $\int_X e^{h} \omega_X^n =\int_X \omega_X^n$. 

It is classical (see 
\cite{Tian97}) that one
can find, for all $0 \leq t<1$, a K\"ahler form $\omega_t=\omega_X+dd^c u_t$ such that
$\Ric \omega_t=t \omega_t+(1-t) \omega_X$, but the situation degenerates as $t \rightarrow 1$.
More precisely
$$
(\omega_X+dd^c u_t)^n=e^{-t u_t+h} \omega_X^n
$$
for an appropriate normalization of $u_t$, and the non-existence of a K\"ahler-Einstein metric
is equivalent to  $\lim_{t\to 1} \mathrm{Osc}_X u_t \rightarrow +\infty$.

Let us set $v_t=u_t -\int_X u_t \om_X^n$; we claim that $\lim_{t\to 1} \|v_t\|_\infty =+\infty$, although it follows from Myers theorem that ${\rm diam}(X,\omega_t) \le c_n$ when $t$ is away from zero since $\Ric \omega_t \ge t \omega_t$. The claim can be checked as follows. Since $v_t$ has zero mean, we have $\sup_X v_t \le C$, hence $-\inf_{X} v_t\ge \mathrm{Osc}(v_t)-C$. Since $\mathrm{Osc}(v_t)=\mathrm{Osc}(u_t)$ diverges, we obtain the claim.
\end{exa}

    \subsubsection{Orlicz integrability}

Theorem \ref{thm:open} allows one to extend and clarify 
 a method of Fu-Guo-Song 
 which provides an upper bound on diameters under \CK
and a lower bound on the Ricci curvature.

\begin{cor} \label{cor:fgs}
Let $(X,\omega_X)$ be a  compact K\"ahler manifold.
Assume that a function $\f \in \PSH(X,\omega_X) \cap {\mathcal C}^{\infty}(X)$ satisfies
$
(\omega_X+dd^c \f)^n=e^{\lambda \f} f dV_X,
$
where 
$dV_X$ is a smooth volume form, $\lambda \in \R^+$, and 
\begin{enumerate}
\item $\omega_X \ge \theta$ with $\theta$ smooth, semi-positive and big (i.e. $\int_X \theta^n >0$);
\item ${\rm Ric} (\omega) \ge -A \omega-B\om_X$ with $\omega:=\omega_X+dd^c \f$;
\item $\mu=fdV_X$ satisfies \CK with weight $w$ such that $\int_X w \circ f dV_X \le C$.
\end{enumerate}
Then ${\rm diam}(X,\omega) \le D$, where $D$ depends
on $A,B,C, w,\theta,dV_X,\lambda$. 
\end{cor}

This result is established in \cite{FGS20} when $B=0$ and 
$\int_X f^{1+\e} dV_X \le C$. 
An extension has been proposed in \cite{GS22} 
to the case when $B=0$ and $\int_X f (\log f)^{n+\e} dV_X \le C$; however, it follows from Theorem \ref{thm:open} that the two conditions are equivalent.

\begin{proof}
When $\mu=fdV_X$ satisfies \CK, it follows from \cite{Kol98,EGZ09}
that $\|\f_{\omega}\|_{\infty}$ is uniformly bounded.
Thus Corollary \ref{cor:fgs} is a consequence of Theorem \ref{thm:open}.
\end{proof}

 \section{The main result} \label{sec:main}
 
In this section we prove our main result Theorem \ref{thmB}.

 \subsection{The special case of the projective space}
 
 To explain the main idea in a simpler setting, we first treat the case of the projective space
 $X=\C\PP^n$. In this case $H^{1,1}(X,\R)=\R\{\omega_{\rm FS}\}$ is one-dimensional and
 the K\"ahler cone is the half-line $\R^+_*\{\omega_{\rm FS}\}$  generated by the Fubini-Study K\"ahler class.
 A rescaling argument thus reduces Theorem \ref{thmB} to the following:

 \begin{thm} \label{thm:projective}
 Let $\omega=\omega_{\rm FS}+dd^c \f$ be a K\"ahler form cohomologous to the Fubini-Study form
 $\omega_{\rm FS}$ on $\C\PP^n$.
 Set  $f_{\omega}=\omega^n/\omega_{\rm FS}^n$ and fix $A>0$ and $p>2n$.
 There exists $C(A,p)>0$ such that if
 $$
 \int_X f_{\omega} |\log f_\omega|^n (\log\circ\log (f_{\omega}+3))^{p}  \omega_{\rm FS}^n \le A,
 $$
 then ${\rm diam}(X,\omega) \le C(A,p)$.
More precisely, for any $\gamma<p-2n$ there exists a constant $C=C(A,p,\gamma)$ such that for two points $x,y\in X$, we have
\[d_\omega(x,y) \le \frac{C}{(\log |\log d_{\omega_X}(x,y)|)^{\frac{\gamma}{2}}}.\]
 \end{thm}
 
 \begin{rem}
 Yang Li's observation yields a uniform bound on the diameter if 
 $\int_X f_{\omega} |\log f_{\omega}|^{p} dV_X <+\infty$ with $p>3n$,
as follows from Theorem \ref{thm:modulus1} and Proposition \ref{pro:diameterYLi}.
Example \ref{exa:radial}.3 shows that our Orlicz condition 
on the density is almost optimal.
 \end{rem}
 
 \begin{rem} \label{rem:lip}
The proof will show, more precisely, that any function
$f$ on $X$ which is $1$-Lipschitz with respect to $\omega$
satisfies $\mathrm{Osc}(f)\le C(X,[\omega])$. This is however equivalent to the above statement:
consider indeed $\gamma$ a geodesic joining two points $x$ and $y$, then
 \[|f(x)-f(y)|= \left|\int_0^1 df(\gamma'(t)) dt\right| \le \int_0^1 |\nabla f|_\omega |\gamma'(t)|_\omega dt \le d_\omega(x,y).\]
 \end{rem}

 \begin{proof}
Let $x=[x_0: \ldots :x_n]$ denote the homogeneous coordinates of $X:=\C\PP^n$ and set 
$$
g_x(y)=g(x,y)=\log \frac{|x \wedge y|}{|x|\cdot |y|}
\in \PSH(X,\omega_{\rm FS}) .
$$
 This "Green function"  $g_x$ is non-positive, smooth away from $x$ and its complex Monge-Amp\`ere measure is well-defined, satisfying
\begin{equation} \label{eq:dirac}
(\omega_{\rm FS}+dd^c_y g_x)^n=\delta_x=\text{Dirac mass at } x.
\end{equation}

\bigskip

\noindent {\bf Step 1. The diameter bound.}

\smallskip

 Fix $x_0 \in X$ and set $f(x)=d_{\omega}(x,x_0)$. This is a $1$-Lipschitz function
 with respect to $d_{\omega}$, hence it is differentiable a.e. with norm at most one; in particular 
 $0 \le df \wedge d^c f \le \omega$.
 Multiplying \eqref{eq:dirac} by $f$ and integrating, we obtain for all $x \in X$,
 $$
 f(x)=\int_X f (\omega_{\rm FS}+dd^c g_x)^n
 \; \; \text{ and } \; \;
 0=\int_X f (\omega_{\rm FS}+dd^c g_{x_0})^n,
 $$
 so that
 $$
 f(x)=\int_X f \left[ (\omega_{\rm FS}+dd^c g_x)^n-\omega_{\rm FS}^n \right]
  -\int_X f \left[ (\omega_{\rm FS}+dd^c g_{x_0})^n-\omega_{\rm FS}^n \right].
 $$
 
 It thus suffices to establish   bounds on 
 $\int_X f \left[ (\omega_{\rm FS}+dd^c g_x)^n-\omega_{\rm FS}^n \right]$ which are uniform in $x \in X$.
 Observe that  $(\omega_{\rm FS}+dd^c g_x)^n-\omega_{\rm FS}^n =dd^c g_x \wedge \sum_{k=0}^{n-1} \omega_{g_x}^k \wedge \omega_{\rm FS}^{n-1-k}$, where $\omega_{g_x}:=\omega_{\rm FS}+dd^c g_x$.
Stokes theorem and  Cauchy-Schwarz inequality yield
 {\small
 \begin{eqnarray*}
 \lefteqn{
 \left| \int_X f dd^c g_x \wedge \omega_{g_x}^k \wedge \omega_{\rm FS}^{n-1-k} \right|
 = \left| \int_X d f \wedge d^c g_x \wedge \omega_{g_x}^k \wedge \omega_{\rm FS}^{n-1-k} \right|} \\
 & \le & \left| \int_X \chi'' \circ g_x \, d g_x \wedge d^c g_x \wedge \omega_{g_x}^k 
 \wedge \omega_{\rm FS}^{n-1-k} \right|^{1/2}
 \left| \int_X (\chi'' \circ g_x)^{-1} \, d f \wedge d^c f \wedge \omega_{g_x}^k \wedge \omega_{\rm FS}^{n-1-k} \right|^{1/2}.
 \end{eqnarray*}
 }
 Here $\chi:\R^- \rightarrow \R$ is a smooth convex increasing function to be chosen below, whose growth
 is slightly smaller than $t$ at $-\infty$ so that $\chi'(-\infty)=0$.
 We normalize $\chi$ so that $\chi'(0) \le 1$ and observe that
 $\chi \circ g_x \in \PSH(X,\omega_{\rm FS})$, with
 \begin{eqnarray*}
  \omega_{\rm FS}+dd^c \chi \circ g_x &=& \chi'' \circ g_x \, dg_x \wedge d^c g_x
 +\chi' \circ g_x \omega_{g_x}+[1-\chi' \circ g_x] \omega_{\rm FS} \\
 &\ge & \chi'' \circ g_x \, dg_x \wedge d^c g_x \ge 0.
  \end{eqnarray*}
  We infer, using Stokes theorem,
  $$
  0 \le \int_X \chi'' \circ g_x \, d g_x \wedge d^c g_x \wedge \omega_{g_x}^k \wedge \omega_{\rm FS}^{n-1-k}
  \le \int_X \omega_{\chi \circ g_x}  \wedge \omega_{g_x}^k \wedge \omega_{\rm FS}^{n-1-k}
  =\int_X \omega_{\rm FS}^n=1.
  $$
  
  \medskip
  \noindent
  To bound the second integral we recall that $df \wedge d^c f \le \omega$ so that
   for all $0\le k \le n-1$, 
  \begin{equation*}
   \int_X (\chi'' \circ g_x)^{-1} \, d f \wedge d^c f \wedge \omega_{g_x}^k \wedge \omega_{\rm FS}^{n-1-k} \le \int_X (\chi'' \circ g_x)^{-1} \omega   \wedge \omega_{g_x}^k \wedge \omega_{\rm FS}^{n-1-k}
   \end{equation*}
  We set $h(t):=\chi''(t)^{-1}, \psi_x:=h\circ g_x$, 
  \[I_{k,x}:=\int_X \psi_x \, \omega   \wedge \omega_{g_x}^k \wedge \omega_{\rm FS}^{n-1-k}\]
  and we choose $\chi(t)=\frac{t}{\log(B-t)^\gamma}$, where $\gamma>0$ is small enough to be determined later and $B>0$ is a uniform constant large enough
  so that $\chi$ is convex increasing on $\R^-$ with $\chi'(0) \le 1$. We have 
      \begin{equation}
    \label{psi}
dd^c \psi_x = h'\circ g_x \, dd^c g_x+h''\circ g_x \, dg_x\wedge d^c g_x 
  \end{equation}
  A  straightforward computation shows that when $t\sim +\infty$, one has
  \begin{equation}
  \label{h}
  h(-t) \sim t (\log t)^{1+\gamma}, \quad  h'(-t)\sim  (  \log t)^{1+\gamma}, \quad h''(-t) \sim \frac{(\log t)^\gamma}{t}
  \end{equation}
  A direct computation (see  \cite[Lemma 2.8]{DGG23}) shows that 
  \[\omega_{g_x} \le e^{-2g_x}\omega_{\rm FS}.\]
  Thanks to the first estimate in \eqref{h}, 
   we see   that   the integral   $\int_X \psi_x  \omega_{g_x}^k \wedge \omega_{\rm FS}^{n-k}$ is uniformly bounded in $x$,
since $k \leq n-1$.
Thus everything comes down to estimating
   \[I'_{k,x}:=\int_X \psi_x dd^c \vp   \wedge \omega_{g_x}^k \wedge \omega_{\rm FS}^{n-1-k}\]
Clearly, the integral above is invariant if we translate $\varphi$ by a constant so we can freely normalize the potential $\varphi$ so that $\varphi(x)=0$. In particular, Theorem~\ref{thm:modulus1} implies that there exist constants $\delta,C>0$ independent of $x,\omega$ such that  
  \begin{equation}
  \label{mod}
  |\varphi| \le \frac{C}{\log(-g_x)^{1+\delta}}.
  \end{equation}
  We are now ready to bound $I'_{k,x}$; we are going to distinguish two cases. \\

  \noindent {\it Case 1. Assume $k<n-1$}. 
    We claim that under this assumption, we have
  \[I'_{k,x}=\int_X \vp dd^c \psi_x \wedge  \omega_{g_x}^k \wedge \omega_{\rm FS}^{n-1-k}. \]
  Fix $\ep>0$ any (small) positive number.
It follows from \eqref{psi} and $\eqref{h}$ that 
there exists $C_{\e}>0$ such that the absolute value of the integral above is bounded 
by 
\[\|\vp\|_\infty \int_X(h'\circ g_x) e^{-2(k+1)g_x} \om_{\rm FS}^n 
\leq C_{\e} \|\vp\|_\infty \int_Xe^{-2(n-1+\ep)g_x} \om_{\rm FS}^n
\leq C_{\e}\frac{2n}{1-\e} \|\vp\|_\infty, \]
as follows from \cite[Lemma 2.8.ii]{DGG23}. 
Now let us justify the integration by parts that we claimed. Set $T_k:=\omega_{g_x}^k \wedge \omega_{\rm FS}^{n-1-k}$; we have the identity
\begin{equation}
\label{IBP 1}
(\vp \partial \bar \partial \psi_x-\psi_x \partial \bar \partial \vp)\wedge T_k = d\left( (\vp \bar \partial \psi_x+\psi_x\partial \vp)\wedge T_k\right). 
\end{equation}  
The current whose differential we are taking in the RHS above is well-defined, as
a wedge product of currents whose coefficients are functions 
whose product is $L^1$. Indeed, \eqref{psi} and \eqref{h} show that $\psi_x \partial \vp \ \wedge T_k$ has coefficients dominated by $e^{-2(k+\ep)g_x}$ while $\bar \partial \psi_x \wedge T_k$ has coefficients dominated by $e^{-2(k+1+\ep)g_x}$, and both are dominated by $e^{-2(n-1+\ep)g_x}$. The claim follows from integrating \eqref{IBP 1} and applying Stokes theorem. \\

\noindent    {\it Case 2. Assume $k=n-1$. }
    Let $g_{x,C}:=\max\{g_x,-C\}$; we have $g_{x,C}\in \PSH(X, \om_{\rm FS})$. Given the definition of the Monge-Ampère operator, it is sufficient to bound
  \begin{equation}
  \label{IBP 2}
I_C:=\int_X \psi_x dd^c \vp \wedge  \omega_{g_{x,C}}^{n-1}=\int_X \vp dd^c \psi_x \wedge  \omega_{g_{x,C}}^{n-1} 
\end{equation}
independently of $C,x$ and $\vp$. Given \eqref{psi}, we can split this integral as follows
\[I_C=\underbrace{\int_X \vp (h'\circ g_x) \om_{g_x}\wedge \omega_{g_{x,C}}^{n-1}}_{=:I_{C,1}}+\underbrace{\int_X \vp (h''\circ g_x) dg_x \wedge d^cg_x \wedge  \omega_{g_{x,C}}^{n-1}}_{=:I_{C,2}}-\underbrace{\int_X \vp (h'\circ g_x) \om_{\rm FS} \wedge \omega_{g_{x,C}}^{n-1}}_{=:I_{C,3}}.\]

Now, we choose $\gamma<\delta$ so that, we have
\begin{equation}
\label{phi h'}
|\varphi \, h'\circ g_x| \lesssim \frac{1}{\log(-g_x)^{\delta-\gamma}} \le M,
\end{equation}
where $M>0$ is some uniform constant. In particular, 
\begin{equation}
\label{C1C3}
I_{C_1}+I_{C_3} \le M \left(\int_{X}\om_{g_x}\wedge \omega_{g_{x,C}}^{n-1}+\int_X \om_{\rm FS} \wedge \omega_{g_{x,C}}^{n-1}\right) = 2M.
\end{equation}
As for the second integral, it follows from \eqref{h} and \eqref{mod} that the density of the integrand with respect to $\omega_{\rm FS}^n$ is dominated by 
\begin{equation}
\label{C2}
\frac{1}{e^{2ng_x}(-g_x)\log(-g_x)^{1+\delta-\gamma}}.
\end{equation}
Using polar coordinates we are reduced to showing that the integral 
$\int_0 \frac{dr}{r\log r (\log(-\log r))^{1+\eta}}$
converges for any $\eta>0$, which is indeed the case (set $u=\log (-\log r)$ and perform the change of variable).

All in all, we have showed that $I_{C}$ is uniformly bounded independently of $C$, hence the first part of the theorem.

\bigskip

\noindent {\bf Step 2. Controlling the distance function.}

\smallskip
We now fix a point $x\in X$ and look at $f(y)=d_\omega(x,y)$ for $y$ close to $x$. We introduce the function $\rho=\rho_{x,y}$ defined by
$\rho:=g_y-g_x$
so that $dd^c \rho = \om_{g_y}-\om_{g_x}$. 
Observe that $|z-y|^2(z_i-x_i)-|z-x|^2(z_i-y_i)=|z-x|^2(y_i-x_i)+(z_i-x_i)(|x-y|^2+2\mathrm{Re} \langle z-x,x-y\rangle$, so that using the triangle inequality $|z-x|\le |z-y|+|x-y|$, we obtain 
\begin{equation*}
\label{diff1}
 \left|\frac{z_i-y_i}{|z-y|^2}-\frac{z_i-x_i}{|z-x|^2}\right|\le \frac{2|x-y|}{|z-x||z-y|}\cdot \left(1+\frac{|x-y|}{|z-y|}\right).
\end{equation*}
hence
\begin{equation}
\label{diff2}
d\rho \wedge d^c \rho \le \frac{C|x-y|^2}{|z-x|^2|z-y|^2}\cdot \left(1+\frac{|x-y|}{|z-y|}\right)^2 \om_{\rm FS}.
\end{equation}

We are going to use the pluricomplex Green function to estimate the distance function $f$. Since $f(x)=0$, we have from \eqref{eq:dirac} 
\begin{eqnarray*}
f(y)&=&\int_X f(\om_{g_y}^n-\om_{g_x}^n)
= \sum_{k=0}^{n-1} \int_X f dd^c\rho\wedge \om_{g_y}^k \wedge \om_{g_x}^{n-1-k}\\
&=&-\sum_{k=0}^{n-1} \int_X df\wedge d^c \rho \wedge \om_{g_y}^k\wedge \om_{g_x}^{n-1-k}.
\end{eqnarray*}
We define $\psi_0=\psi_x=h\circ g_x$, $\psi_{n-1}=\psi_y=h\circ g_y$ and for $1\le k \le n-1$, we set $\psi_k=\psi_{x,y}:=h\circ (g_x+g_y)$ where $h(t)=|t| (\log |t|)^{1+\gamma}$ for $\gamma>0$ small to be determined later (cf \eqref{h}). 
By Cauchy-Schwarz inequality, we have
\begin{equation}
\label{CS}
f(y)\le \sum_{k=0}^{n-1} \Big(\underbrace{\int_X \psi_k^{-1} d\rho \wedge d^c \rho \wedge \om_{g_y}^k\wedge\om_{g_x}^{n-1-k}}_{=:J_k}\Big)^{\frac 12}\Big(\underbrace{\int_X \psi_k\, \omega \wedge \om_{g_y}^k\wedge\om_{g_x}^{n-1-k}}_{=:I_k}\Big)^{\frac 12} 
\end{equation}

\noindent {\it 2.1. Estimating the $J_k$ integrals.}
We split the computation of the integral into two zones. First, we look outside a small neighborhood of $x$ (and $y$, which is assumed to be close to $x$), and then we look at what happens near $x$. 
We next use a local coordinate system $(z_i)$ such that $x=0$ hence $g_x = \log |z|+O(1)$ and $g_y = \log |z-y|+O(1)$. Let $\ep:=|y|$ and let $\bar y := \ep^{-1}y$, which satisfies $\bar y \in  S^{2n-1}$. Finally, we introduce the variable $w=\ep^{-1} z$. \\

$\bullet$ The first, easiest, case is when we look at $J_k$ on $X\setminus (|z|\le C^{-1})$ for $C>0$ large but fixed. There, the function $\rho$ is a smooth function of $z,x,y$ with uniformly bounded derivatives (in all variables). In particular  $d\rho\wedge d^c \rho  \le C |x-y|^2 \om_{\rm FS}$ 
and $\om_{g_x}+\om_{g_y} \le C \om_{\rm FS}$ for some uniform $C$ hence $J_k$ on the considered zone is less than $C|x-y|^{2} \int_X \om_{\rm FS}^n=C\ep^{2}$.\\

$\bullet$ Next, we  estimate $J_k$ over  $(|z|\le C\ep)$ for  $C>0$ large. 

We have to estimate
\[\int_{|z|\le C\ep}\psi_k(z)^{-1} \left(\frac{1}{|z|^2}+\frac{1}{|z-y|^2}\right) \frac{1}{|z-y|^{2k}|z|^{2(n-1-k)}}|dz|^2\]
which is equal to 
\[\int_{|w|\le C}\psi_k(\ep w)^{-1} \left(\frac{1}{|w|^2}+\frac{1}{|w-\bar y|^2}\right) \frac{1}{|w-\bar y|^{2k}|w|^{2(n-1-k)}}|dw|^2.\]
By symmetry, it is enough to consider the two cases $k=0$ and $0<k<n-1$. There are two meaningful zones which are near $w=0$ and $w=\bar y$. Near $w=0$, we have $g_x=\log\ep |w|+O(1)$ and $g_y=\log \ep+O(1)$ hence $\psi_k(z)\ge h(\log \ep |w|)+O(1)$ no matter the value of $k$.   Therefore the integral behaves (uniformly in $\bar y\in S^{2n-1}$) like 
\[\int_{|w|\le \frac 12} \frac{|dw|^2}{|w|^{2(n-k)}(-\log |\ep  w|)(\log(-\log |\ep  w|))^{1+\gamma}}.\]
Clearly, the most singular case is when $k=0$ where the integral becomes
 \begin{equation}
 \label{log log}
 \int_{|w|\le \frac{\ep}{2}} \frac{|dw|^2}{|w|^{2n}(-\log |w|)(\log(-\log |w|))^{1+\gamma}}=O\Big(\frac{1}{\log (-\log \ep)^\gamma}\Big).
 \end{equation}
 Near $w=\bar y$, $\psi_k(z)\ge -\log \ep+O(1)$ hence the integral is dominated by 
\[\int_{|w-\bar y |\le \frac 12} \frac{|dw|^2}{|w-\bar y|^{2(k+1)}(-\log \ep)(\log(-\log \ep))^{1+\gamma}}=O\Big(\frac{1}{\log \ep}\Big)\]
since $k+1<n$. \\

 $\bullet$ Finally, we  analyze $J_k$ in the zone $(C\ep \le |z| \le C^{-1})$, or equivalently $(C \le |w| \le (C\ep)^{-1})$. Here again, we assume that $k<n-1$. Then we have $\psi_k\ge h(\log\ep |w|)+O(1)$. Moreover, it follows from \eqref{diff2} that in the above zone, $J_k$ is dominated by 
{\small \[\int_{C \le |w| \le (C\ep)^{-1}}\frac{1}{\psi_k(\ep w)} \left(\frac{\ep}{\ep^2 |w| |w-\bar y|}\right)^{2}\left(1+\frac{1}{|w-\bar y|}\right)^2\frac{1}{(\ep|w-\bar y|)^{2k}(\ep |w|)^{2(n-1-k)}}\ep^{2n} |dw|^2\]
}
which is controlled by 
\[\int_{C \le |w| \le (C\ep)^{-1}} \frac{|dw|^2}{|w|^{2n+2}(-\log |\ep w|)(\log(-\log |\ep w|))^{1+\gamma}}.\]
Reversing to the variable $z=\ep w$ and using polar coordinates, this integral becomes
\[\ep^2\int_{C\ep \le r \le C^{-1}} \frac{dr}{r^3(-\log r)(\log(-\log r))^{1+\gamma}}.\]
Setting $u=-\log r$, the above quantity can be rewritten as 
\[\ep^2\int_{-\log C \le u \le -\log C\ep} \frac{e^{2u}du}{u(\log u)^{1+\gamma}}.\] \\
For $u$ large enough, the integrand is increasing hence the integral is dominated by 
\[\ep^2(-\log \ep) \frac{e^{-2\log \ep}}{(-\log \ep)(\log (-\log \ep))^{1+\gamma}} \lesssim \frac{1}{\log (-\log \ep)^{1+\gamma}}\]
hence the result.\\

\noindent {\it 2.2. Estimating the $I_k$ integrals.} 
 When bounding the diameter, we have already proved that $I_0$ and $I_{n-1}$ are bounded uniformly, so from now on we assume that $0<k<n-1$. 
 
 \smallskip

$\bullet$ On $X\setminus (|z|\le 1)$, $\psi_k$ is bounded hence on that zone $I_k$ is less than $ C \int_X \omega \wedge\om_{g_y}^k \wedge \om_{g_x}^{n-1-k}= C$. \\

From now on, we look at $I_k':=\int_{|z|\le 1} \psi_k \om \wedge \om_{g_y}^k \wedge \om_{g_x}^{n-1-k}$ and one can assume without loss of generality that $\psi_k$ has compact support in $(|z|\le 1)$. This will allow us to perform integration by parts without having to deal with boundary terms. 
\[I_k'=\underbrace{\int_{|z|\le 1}\psi_k \om_{\rm FS} \wedge \om_{g_y}^k \wedge \om_{g_x}^{n-1-k}}_{=:I_{k,1}'}+\underbrace{\int_{|z|\le 1}\psi_k dd^c\varphi \wedge \om_{g_y}^k \wedge \om_{g_x}^{n-1-k}}_{=:I_{k,2}'}\]

$\bullet$ The integral $I_{k,1}'$. It is dominated by 
\[\int_{|z|\le 1} \frac{|\log^2(|z||z-y|)|dz|^2}{|z-y|^{2k}|z|^{2(n-1-k)}}=\int_{|w|\le \ep^{-1}} \frac{\ep^2\log^2(\ep^2|w||w-\bar y|)|dw|^2}{|w-\bar y|^{2k}|w|^{2(n-1-k)}}.\]
Near $w=0$, we have for any small $\delta>0$ that the integrand is dominated by $\ep^2(\ep^2|w|)^{-\delta} |w|^{-2(n-1-k)} \le  |w|^{-2n+1}$, and convergence follows. The behavior near $w=\bar y$ 
is dealt with similarly. Finally, near infinity, the integrand is dominated by $\ep^{2+2\delta} |w|^{-2(n-1-\delta)}$ while $\int^{\ep^{-1}}\frac{\ep^{2+2\delta} |dw|^2}{ |w|^{2(n-1-\delta)}}=\int^{\ep^{-1}} \ep^{2+2\delta} r^{1+2\delta} dr \simeq 1$; we are done with $I_{k,1}'$.\\

$\bullet$ The integral $I_{k,2}'$. One can assume that $\varphi(x)=0$, hence 
\begin{equation}
\label{mod2}
|\varphi(z)|\le \frac{C}{(\log(-\log |z|))^{1+\delta}}
\end{equation}
for some $\delta>0$. An integration by parts shows that 
$I_{k,2}'=\int_{|z|\le 1}\varphi dd^c \psi_k  \wedge \om_{g_y}^k \wedge \om_{g_x}^{n-1-k}.$
Next, if we set $\sigma:=g_x+g_y$, \eqref{h} and \eqref{mod2} show that if $\gamma <\delta$, one has
\begin{equation}
\label{hp}
|\varphi h'(\sigma)(\ep w)| \le 
\begin{cases} 
C \cdot \frac{ \log(-\log |w-\bar y|)^{1+\gamma}}{\log(-\log \ep)^{1+\delta}} & \mbox{if} \, |w-\bar y| \le \frac 12; \\
C & \mbox{otherwise for } \, |w|\le \ep^{-1}.
\end{cases}
\end{equation}
as well as 
\begin{equation}
\label{hpp}
|\varphi h''(\sigma)(\ep w)| \le 
\begin{cases} 
C & \mbox{if } \, |w|\le C; \\
 \frac{ C }{(-\log \ep |w|)(\log(-\log \ep|w|))^{1+\delta-\gamma}} & \mbox{if} \, C\le |w| \le \ep^{-1}. 
\end{cases}
\end{equation}
Now, let us write
\begin{eqnarray*}
dd^c \psi_k &=& h'(\sigma)(dd^c g_x+dd^c g_y)+h''(\sigma)d(g_x+g_y) \wedge d^c(g_x+g_y)\\
&=&h'(\sigma)(\om_{g_x}+\om_{g_y}-2\om_{\rm FS})+h''(\sigma)d(g_x+g_y) \wedge d^c(g_x+g_y). 
\end{eqnarray*}

As in the previous computations, we split the integral $I_{k,2}'$ into the zones $(|w|\le C)$ and $(C\le |w| \le \ep^{-1})$.  
 On $(|w|\le C)$ we have by \eqref{hpp},
\[ |\varphi h''(\sigma)|d\sigma \wedge d^c\sigma \wedge \om_{g_y}^k \wedge \om_{g_x}^{n-1-k} \le C \Big( \frac{1}{|w|^2}+\frac{1}{|w-\bar y|^2}\Big)\cdot \frac{ |dw|^2}{|w-\bar y|^{2k}|w|^{2(n-1-k)}}\]
whose integral on the considered zone converges since $k>0$ and $k+1<n$. Next, by \eqref{hp} we have on  $(|w|\le C)\cap (|w-\bar y| \ge \frac 12)$ the following inequality
\[ |\varphi h'(\sigma)|(\om_{g_x}+\om_{g_y}+\om_{\rm FS}) \wedge \om_{g_y}^k \wedge \om_{g_x}^{n-1-k}\le  \frac{ C|dw|^2}{|w|^{2(n-k)}}\] 
and the integral of the RHS is clearly convergent on the latter domain. On $(|w-\bar y| \ge \frac 12)$, \eqref{hp} shows that we have
\[ |\varphi h'(\sigma)|(\om_{g_x}+\om_{g_y}+\om_{\rm FS}) \wedge \om_{g_y}^k \wedge \om_{g_x}^{n-1-k}\le \frac{C \log(-\log |w-\bar y|)^{1+\gamma}|dw|^2}{|w-\bar y|^{2(k+1)}}\]
and the convergence follows since $k+1<n$. Combining those three inequalities shows that $I_{k,2}'$ is uniformly convergent over $(|w|\le C)$. \\

For the zone $(C\le |w|\le (C\ep)^{-1})$ we use \eqref{hpp} again to obtain
\[ 
|\varphi h''(\sigma)|d\sigma \wedge d^c\sigma \wedge \om_{g_y}^k \wedge \om_{g_x}^{n-1-k} \le  \frac{C|dw|^2}{|w|^{2n}(-\log \ep |w|)(\log(-\log \ep|w|))^{1+\delta-\gamma}}.
\]
The integral of the RHS is the same as $\int_{\ep C \le |z| \le C^{-1}}\frac{|dz|^2}{|z|^{2n}(-\log |z|)(\log(-\log |z|))^{1+\delta-\gamma}}$ which is uniformly bounded as $\ep\to 0$ since $\gamma<\delta$. Finally, we have thanks to \eqref{hp}
\[ |\varphi h'(\sigma)|(\om_{g_x}+\om_{g_y}+\om_{\rm FS}) \wedge \om_{g_y}^k \wedge \om_{g_x}^{n-1-k} \le C(\om_{g_x}+\om_{g_y}+\om_{\rm FS}) \wedge \om_{g_y}^k \wedge \om_{g_x}^{n-1-k} \]
and the integral of the RHS can be computed in cohomology; it is equal to $3$. This ends to show that $I_{k,2}'$ is uniformly bounded, hence the theorem is proved. 
 \end{proof}

 \subsection{General case}

 We now deal with the general case whose statement we recall:

\begin{thm}
\label{thm:general}
Let $(X,\om_X)$ be a compact K\"ahler manifold of complex dimension $n$ and let
$dV_X:=\frac{\omega_X^n}{n!}$ the associated smooth volume form.
Let ${\mathcal K} \subset H^{1,1}(X,\R)$ be a compact subset of the K\"ahler cone of $X$.
Given $\omega$ a K\"ahler form in ${\mathcal K}$, we set $f_{\omega}=\omega^n/dV_X$
and consider, for $A>0$ and  $p>2n$ fixed,
$$
{\mathcal K}_{A,p}:=\{ \omega \in  {\mathcal K} \text{ such that } 
\int_X f_{\omega} |\log f_{\omega}|^{n} (\log(\log f_\om+3))^p  dV_X \le A \}.
$$
 For any $\gamma<p-2n$, there exists a constant $C=C({\mathcal K},dV_X,A,p,\gamma)>0$ such that for all $\omega \in {\mathcal K}_{A,p}$ and any two points $x,y\in X$, we have 
 \[d_\omega(x,y) \le \frac{C}{(\log |\log d_{\omega_X}(x,y)|)^{\frac{\gamma}{2}}}.\]
  \end{thm}
 
 The proof is similar to that of Theorem \ref{thm:projective},
  but it requires to control extra terms due to the fact that
 we can only construct an approximate Green function for the complex Monge-Amp\`ere
 operator (see \cite{CG09}).

\begin{proof}
We use a double cover of $X$ by finitely many open charts $U_j \subset V_j$ and fix
$\chi_j \in {\mathcal C}^{\infty}_c(V_j)$ such that  $0 \le \chi_j \le 1$
with $\chi_j \equiv 1$ near $\overline{U}_j$.

For $x \in X$, we let $\ell_{j,x}(z)$ denote the function $\log |z-x|-C$ in the chart $V_j$,
where $z$ denotes holomorphic coordinates,
and set $g_x(z)=\chi_j \ell_{j,x}(z)$. If $x \in X$ belongs to several charts,
we choose one index $j=j(x)$ so that $g_x$ is uniquely determined.
The constant $C>0$ is chosen so that $g_x \le -1$ on $X$.

The function $g_x$ is psh near $x$ and smooth in $X \setminus \{x\}$, hence it is quasi-psh.
Moreover 
$$
dd^c g_x \ge  d \chi_j \wedge d^c \ell_{j,x}
+d \ell_{j,x} \wedge d^c \chi_j +\ell_{j,x} dd^c \chi_j \ge -C_2 \omega_X
$$
for a constant $C_2>0$ independent of $x$. 

Rescaling $\omega_X$ we can assume that $C_2=1$,
so that each function $g_x$ belongs to $\PSH(X,\omega_X)$. We set $\omega_{g_x}:=\om_X+dd^c g_x$. 
The complex Monge-Amp\`ere measure of this approximate Green function is well-defined
and satisfies
\begin{eqnarray}
\label{om gx}
(\omega_X+dd^c g_x)^n&=&\om_X^n+(dd^c g_x)^n+\sum_{k=1}^{n-1}{n \choose k} \omega_X^k \wedge (dd^c g_x)^{n-k} \\
&=& \om_X^n+\delta_x+\Theta_x+\sum_{k=1}^{n-1}{n \choose k} \omega_X^k \wedge (\om_{g_x}-\om_X)^{n-k} \nonumber \\
&=&\omega_X^n+\delta_x+\Theta_x+dd^c g_x \wedge \omega_X \wedge S_x, \nonumber
\end{eqnarray}
where 
\begin{itemize}
\item[$\bullet$] $\delta_x$ is the Dirac measure at point $x$,
\item[$\bullet$] $\Theta_x$ is a smooth measure on $X$,
\item[$\bullet$] and $S_x=\sum_{\ell=0}^{n-2} c_\ell (\omega_X +dd^c g_x)^\ell \wedge \omega_X^{n-2-\ell}$,
for some $c_j \in \R$ independent of $x$.
\end{itemize}

As in the proof of Theorem \ref{thm:projective}, we consider 
$f(x)=d_{\omega}(x,x_0)$, where $x_0$ is a fixed base point.
This is a $1$-Lipschitz function with respect to $d_{\omega}$, hence
it s differentiable almost everywhere with
$0 \le df \wedge d^c f \le \omega$. Observe that for all $x \in X$,
\begin{equation}
\label{kernel}
f(x)=\int_X f (\omega_X+dd^c g_x)^n
-\int_X f \omega_X^n-\int_X f \Theta_x-\int_X f dd^c g_x \wedge \omega_X \wedge S_x.
\end{equation}
For $x=x_0$ we obtain $f(x_0)=0$, hence
\begin{eqnarray*}
f(x) &=& \int_X f [(\omega_X+dd^c g_x)^n-\omega_X^n]
-\int_X f [(\omega_X+dd^c g_{x_0})^n-\omega_X^n] \\
&+& \int_X f dd^c g_{x_0} \wedge \omega_X \wedge S_{x_0}
-\int_X f dd^c g_x \wedge \omega_X \wedge S_x+\int_X f[\Theta_{x_0}-\Theta_{x}].
\end{eqnarray*}

We are going to bound each of these integrals from above, this will provide a uniform bound on $f$
hence on the diameter of $(X,\omega)$. Decomposing
\begin{equation}
\label{difference}
(\omega_X+dd^c g_x)^n-\omega_X^n=dd^c g_x \wedge \sum_{k=0}^{n-1} \omega_{g_x}^k \wedge \omega_X^{n-1-k},
\end{equation}
controlling the first four integrals reduces to establish a uniform bound on
integrals of the form 
$\int_X f dd^c g_x \wedge  \omega_{g_x}^k \wedge \omega_X^{n-1-k}$.
This can be treated as in the proof of Theorem \ref{thm:projective} by
integrating by parts and using a twisted Cauchy-Schwarz inequality.

By comparison with the proof of Theorem \ref{thm:projective}, we
need to see that the bounds on all integrals involved
are uniform with respect to $\omega \in {\mathcal K}$. 
Using Lemma \ref{lem:goodref}  we
decompose $\omega=\omega_{\alpha}+dd^c \f$, where 
$C^{-1} \omega_X \le \omega_{\alpha} \le C \omega_X$
for some uniform constant $C=C({\mathcal K})>0$. The salient point is then to
establish a uniform bound on the modulus of continuity of $\f$.
The latter follow from Theorem \ref{thm:modulus1} and Remark \ref{rem:modulus}.

\smallskip

We finally focus on the last term $\int_X f[\Theta_x-\Theta_{x_0}]$.
The smooth $(n,n)$-form $\Theta_x-\Theta_{x_0}$ is $dd^c$-cohomologous to 
$\delta_{x_0}-\delta_x$, hence it can be written as $dd^c \eta(x,x_0)$
for some smooth form $\eta$ of bidegree $(n-1,n-1)$ defined
by the Green operator of Hodge theory (see \cite{GH78}). 
In particular  $\eta$ depends smoothly on $x,x_0$.
Using a partition of unity, we can   decompose 
$$
\eta=\sum_{\ell,I,J} \eta_{\ell,I,\overline{J}} \beta_{\ell,I,\overline{J}},
$$
where $\eta_{\ell,I,\overline{J}}$ are smooth functions and 
$\beta_{\ell,I,\overline{J}}$ is a local basis of smooth forms of bidegree $(n-1,n-1)$
which are {\it closed} and {\it positive}. Integrating by parts we thus obtain
\begin{eqnarray*}
\int_X f[\Theta_x-\Theta_{x_0}]
&=& \sum_{\ell,I,J} \int_X f dd^c \eta_{\ell,I,\overline{J}} \wedge \beta_{\ell,I,\overline{J}} 
= \sum_{\ell,I,J} -\int_X df \wedge d^c \eta_{\ell,I,\overline{J}} \wedge \beta_{\ell,I,\overline{J}}\\
&\le &  \sum_{\ell,I,J}
\left( \int_X df \wedge d^c f \wedge  \beta_{\ell,I,\overline{J}} \right)^{1/2}
\left( \int_X d \eta_{\ell,I,\overline{J}} \wedge d^c \eta_{\ell,I,\overline{J}} \wedge  \beta_{\ell,I,\overline{J}} \right)^{1/2} \\
&\le &  C \sum_{\ell,I,J}
\left( \int_X \omega \wedge  \beta_{\ell,I,\overline{J}} \right)^{1/2} 
\le  C' \left( \int_X \omega \wedge  \omega_X^{n-1} \right)^{1/2},
\end{eqnarray*}
bounding each form $\beta_{\ell,I,\overline{J}}$ by a uniform multiple of 
the reference form $\omega_X^{n-1}$.
 Observe finally that $\int_X \omega \wedge  \omega_X^{n-1}$ remains uniformly bounded from above
as the cohomology class $[\omega]$ evolves in a bounded subset of $H^{1,1}(X,\R)$.

\smallskip

The more precise bound
$d_\omega(x,y) \le C(\log |\log d_{\omega_X}(x,y)|)^{-\frac \gamma 2}$
is proved similarly.
  \end{proof}

\section{Diameter bounds for currents}

In this section, we prove Theorem \ref{thmC} and
adopt the following setup. 

\begin{setup}
\label{setup}
Let $(X,\om_X)$ be a compact Kähler manifold and let $E$ be a divisor with simple normal crossings. Let $X^\circ:=X\setminus E$ and let $T=\theta+dd^c\varphi$ be a closed positive $(1,1)$-current where $\theta$ is a smooth closed $(1,1)$-form. We assume that $\omega:=T|_{X^\circ}$ is a Kähler form. Given $x\in X^\circ$, we define $\rho_x(y)=d_\omega(x,y)$ for any $y\in X^\circ$. 
\end{setup}

\subsection{Integrability properties of the distance}

The following result builds upon \cite[Lemma~1.3]{DPS}.

\begin{lem}
\label{dist L2}
In the Setup~\ref{setup}, the function $\rho_x$ belongs to $W^{1,2}(X, \omega_X^n)$ for any $x\in X^\circ$.
\end{lem}

\begin{proof}
We proceed in two steps. 

\medskip

\noindent {\it Step 1:  $\rho_x\in L^2(X)$}. 
This step is very similar to \cite[Lemma~1.3]{DPS}, and we will explain how to deal with the singularities near $E$. 
We pick two open neighborhoods $E\subset U \Subset U'$ of $E\subset X$ such that $U'=\cup_{i=1}^N U_i'$ is covered by finitely many coordinate charts $U_i'\simeq \Delta^n$ where $\Delta:=\{z\in \mathbb C; |z|<2\}$ and $E\cap U_i'$ is a union of hyperplanes. Next, on can assume that each $\overline U_i:=\overline U\cap U_i'$ is  the convex set $K:=[-1, 1]^{2n}$ under the natural embedding $[-1,1]^2\subset \Delta$. We let $C>0$ be such under the suitable identifications, we have $C^{-1} \omega_{\rm eucl} \le \omega_X|_{U_i'} \le C \omega_{\rm eucl}$ where $\omega_{\rm eucl}$ is the standard euclidean metric on $\mathbb C^n$.

Now, let us pick $x_0\in X^\circ$.Clearly, $\rho_{x_0}$ is bounded on $X\setminus U$, so it is enough to show that ${\rho_{x_0}}|_{U_i}$ is $L^2$ for each $i$. From now on, we fix $i$ and we work exclusively on $U_i\simeq K$  . Given $x,y\in K$, consider the line $L_{x,y}=\{tx+(1-t)y, t\in [0,1]\}$. We claim that the set $S:=\{(x,y)\in K\times K; L_{x,y}\cap E \neq \emptyset\}$ is included in a finite union of proper real analytic subvarieties of $K\times K$. Indeed, if we write $x=(u_1, v_1, \ldots, u_n, v_n)$ and let $k$ be such that $E\cap U'_i=(z_1\cdots z_k=0)$, then 
\[S\subset (E\times K) \cup (K\times E) \cup \bigcup_{j=1}^k \{(x,x')|u_jv'_j-v_ju'_j=0\},\]
which shows the claim. 


Let $dV^{\otimes 2}= p_1^*\omega_{\rm eucl}^n\otimes p_2^*\omega_{\rm eucl}^n$ (and similarly for $dV_X$) be the product volume form on $K\times K$ where $p_i$ are the two projections onto $K$.  We have $\int_S dV^{\otimes 2}=0$. For $x,y\in K\times K\setminus S$, we have $d_\omega(x,y)\le \ell_\omega(L_{x,y})$ where $\ell_\omega$ means the length with respect to $\omega$. 

Therefore, we have
\begin{equation}
\label{comp int}
\int_{K\times K} d_\omega^2(x,y)dV^{\otimes 2}(x,y) \le \int_{K\times K} \ell_{\omega}^2(L_{x,y}) dV^{\otimes 2}(x,y)
\end{equation}
where both integrals are unambiguously defined thanks to the previous observation. As a side note, we should say that the choice of $S$ as before makes the presentation a bit cleaner but we could just as well have worked with $S=E\times E$ since for $(x,y)\notin E\times E$, the line $L_{x,y}$ meets $E$ finitely many times so we can still define $\ell_\omega(L_{x,y})$. 

Using Cauchy-Schwarz inequality it suffices to bound from above the   integral
\[I=\int_{K\times K}\int_{0}^1 \omega_{tx+(1-t)y}(x-y)dtdV^{\otimes 2}(x,y) , \]
where $x-y$ is seen as a (constant) tangent vector. Let $h=\mathrm{tr}_{\om_{\rm eucl}} \omega$. Using the elementary inequality $\omega \le \mathrm{tr}_{\om_{\rm eucl}} \omega \cdot \om_{\rm eucl}$, we get
\begin{eqnarray*}
I &\le & \int_{0}^1 dt \int_{K\times K} \|x-y\|_{\rm eucl}^2 h(tx+(1-t)y)dV^{\otimes 2}(x,y)\\
&\le & 4\int_{0}^1 dt \int_{K\times K} h(tx+(1-t)y)dV^{\otimes 2}(x,y).
\end{eqnarray*}
Now, given $x\in K$ and $t\in [0,\frac{1}{2}]$, consider the linear isomorphism $u:\mathbb C^n\to \mathbb C^n$ defined by $u(y)=tx+(1-t)y$. We have $u(K)\subset K$ hence the change of variable yields 
\[\int_0^{1/2} dt \int_{y\in K}h(tx+(1-t)y) \om_{\rm eucl}^n(x) \le 2^{2n-1}\int_K h(u)\om_{\rm eucl}^n(u).\]
The symmetric operation  yields
$I\le 2^{2n} \mathrm{vol}_{\rm eucl}(K) \int_Kh\, \om_{\rm eucl}^n$
when $t\in [\frac 12 ,1]$.
Now, since $\mathrm{tr}_{\om_{\rm eucl}} \omega \cdot \omega_{\rm eucl}^n=n\omega \wedge \omega_{\rm eucl}^{n-1}$, we obtain 
$$
I \le  n2^{4n} \int_K\omega \wedge \omega_{\rm eucl}^{n-1}
\le  n2^{4n}C^{n-1} \int_X T \wedge \om_X^{n-1}.
$$
By \eqref{comp int}, this implies that
\[\int_{U\times U}d_\omega^2(x,y)dV_X^{\otimes 2}(x,y) \le n2^{4n}C^{3n-1}N \int_X T \wedge \om_X^{n-1},\]
hence $\rho_x$ is in $L^2(\omega_X^n)$ 
for almost every $x\in X^\circ$. 
Now, given two points $x,x' \in X\setminus E$, the triangle inequality yields 
$|\rho_x-\rho_{x'}|\le d_\omega(x,x')$
 hence  $\rho_x$ is in $L^2(\omega_X^n)$ for every $x\in X^\circ$. \\

\medskip

\noindent {\it Step 2: $d\rho_x\in L^2(X)$.} 
This amounts to showing that the current $d\rho_x$ does not charge $E$. It is standard (see e.g. \cite[\textsection  9]{CGP} )
to construct a family of cut-off functions $(\chi_\delta)$ with values in $[0,1]$ such that 
\begin{enumerate}[label=$\bullet$]
\item $\chi_\delta \equiv 1$ near $E$ and $\mathrm{Supp}(\chi_\delta)$ converges to $E$ when $\delta\to 0$.
\item $d\chi_\delta\wedge d^c \chi_\delta \le \omega_P$ where $\omega_P$ is a Kähler metric on $X\setminus E$ with Poincaré type growth near $E$. 
\end{enumerate}
 Let $S_\delta$ be the support of $d\chi_\delta$. Let $\alpha$ be any smooth $(2n-1)$-form on $X$; up to replacing $\omega_X$ with a large multiple, we have
\begin{eqnarray*}
\left|\int_X \chi_\delta d\rho_x \wedge \alpha \right| & \le & \int_{S_\delta} \rho_x \omega_X^{n}+\left|\int_X \rho_x d\chi_\delta\wedge \alpha\right| \\
  &\le & \int_{S_\delta} \rho_x\omega_X^n+\int_{S_\delta}  \rho_x (\mathrm{tr}_{\omega_X} \omega_P)^{\frac 12} \omega_X^n \\
  &\le &  \int_{S_\delta} \rho_x\omega_X^n+ \int_{S_\delta} \rho_x^2\omega_X^n+ \int_{S_\delta} n\omega_P\wedge \omega_X^{n-1}.
\end{eqnarray*} 
Since $\rho_x\in L^2(\omega_X^n)$ by the first step and $\omega_P\wedge \omega_X^{n-1}$ puts no mass on $E$, 
the RHS converges to zero when $\delta \to 0$, hence the lemma.
\end{proof}

Combining Lemma~\ref{dist L2} and Yang Li's results \cite[Theorem~4.1]{Li20}, we can drastically improve the integrability properties of $\rho_x$: 

\begin{lem}
\label{exp}
In the Setup~\ref{setup}, there exists $\alpha>0$ such that for any $x\in X^\circ$, 
$e^{\alpha \rho_x} \in L^1(X, \omega_X^n)$.  
\end{lem}

\begin{proof}
The proof is identical to that of \cite[Theorem~4.1]{Li20}, so we'll only briefly sketch it. Given a small ball of radius $r$ in $X$, it follows from Lelong's results that the masses $\frac{1}{r^{2n-2}}\int_{B(r)} T\wedge \om_X^{n-1}$ increase with $r$ so that
\[\int_{B(r)} T \wedge \omega_X^{n-1} \le C r^{2n-2}\]
 where $C$ depends only on $X$ and the cohomology class of $T$. On $X^\circ$, we have $|d\rho_x|^2_{\omega_X} \le \mathrm{tr}_{\om_{X}}\om \cdot |d\rho_x|^2_\omega= \mathrm{tr}_{\om_{X}}\om$ hence $\int_{B(r)} |d\rho_x|^2_{\omega_X} \om_X^n \le Cr^{2n-2}$. Since $d\rho_x \in L^2(X)$ by Lemma~\ref{dist L2}, the previous inequality holds for any balls (even those centered on $E$) and we can appeal to the Poincaré inequality for small balls on $(X,\om_X)$  to infer that $\rho_x$ has bounded mean oscillation, hence John-Nirenberg inequality yields $\alpha>0$ such that
 \[\int_Xe^{\alpha(\rho_x-\int_X \rho_x \om_X^n)}\om_X^n<+\infty,\]
which concludes the proof.
\end{proof}

\subsection{Control on the diameter}

So far we have assumed nothing on the regularity of $T$ near $E$. 
If the potential $\varphi$ of $T$ is continuous and  its modulus of continuity $m_\f$ with respect to $d_{\om_X}$ satisfies the condition $m_1(r):=\int_0^r t^{-1}\sqrt{m_{\varphi}(t)} dt<+\infty$ (e.g. if $\varphi$ is Hölder continuous), then the proof of Proposition~\ref{pro:diameterYLi} goes through verbatim and shows that $d_\omega\le m_1 \circ d_{\om_X}$ on $X^\circ \times X^\circ$. In particular, $\mathrm{diam}(X^\circ,\omega)<+\infty$. 

The next result provides a sharp criterion to guarantee that this diameter is finite.

\begin{thm}\label{thm:diameter general}
In the Setup~\ref{setup}, assume that there exists $\delta>0$ such that the modulus of continuity $m_\varphi$ of $\varphi$ satisfies $m_\varphi(r)\le \frac{C}{(\log(-\log r))^{1+\delta}}$ for some $C>0$. Then we have 
\[\mathrm{diam}(X^\circ, \omega)<+\infty.\]
\end{thm}

\begin{proof}
The proof is essentially the same as that of Theorems~\ref{thm:projective} and \ref{thm:general} but we need to treat the integrations by parts carefully here. Let us fix a point $x_0\in X^\circ$ and consider $f:=\rho_{x_0}$. We know that $f$ is Lipschitz on $X^\circ$ and satisfies $df\wedge d^c f \le \omega$ on that open set. Thanks to Lemma~\ref{dist L2}, we know that $f\in W^{1,2}(X)$; in particular, the previous inequality extends to the inequality 
\begin{equation}
\label{weak gradient}
df\wedge d^c f \le T \quad {\rm weakly \, on \,\,\, } X.
\end{equation}

Given \eqref{kernel} and \eqref{difference}, we have for all $x\in X^\circ$
\begin{equation}
\label{kernel2}
f(x)=\sum_{k=0}^{n-1}c_k\int_X fdd^c g_x\wedge \omega_{g_x}^k \wedge \omega_X^{n-1-k}-\int_X f \Theta_x
\end{equation}
where $c_k\in \mathbb R$ and $\Theta_x$ is a smooth $(n,n)$-form varying continuously with $x$. 
In particular, there exists $C>0$ independent of $x$ such that $\pm\Theta_x \le C \omega_V^n$ so that $\left| \int_X f \Theta_x\right| \le C \|f\|_{L^1(\omega_X^n)}$. We are left to controlling the integrals 
\[J_{k,x}:=\int_X fdd^c g_x\wedge \omega_{g_x}^k \wedge \omega_X^{n-1-k}.\]
Note that the current $S_k:=fd^cg_x\wedge \omega_{g_x}^k \wedge \omega_X^{n-1-k}$ is well-defined for all $k$. 
Now $\int_X dS_k=0$, hence
\begin{equation}
\label{stokes}
J_{k,x}   =- \int_X d f \wedge d^c g_x \wedge \omega_{g_x}^k \wedge \omega_{X}^{n-1-k}.
\end{equation}
Since $df\in L^2$ we can use Cauchy-Schwarz and \eqref{weak gradient} to obtain
\[|J_{k,x}|\le  \left| \int_X \chi'' \circ g_x \, d g_x \wedge d^c g_x \wedge \omega_{g_x}^k  \wedge \omega_{X}^{n-1-k} \right|^{1/2}\left| \int_X (\chi'' \circ g_x)^{-1} \,T \wedge \omega_{g_x}^k \wedge \omega_{X}^{n-1-k} \right|^{1/2}\]
for any smooth convex function $\chi$. 

The proof of Theorem~\ref{thm:projective} then goes through without a change, once we observe that one can perform the integration by parts \eqref{IBP 1} and \eqref{IBP 2} just like in the smooth case. Indeed, with the notation of \eqref{IBP 1}, $T_k$ and $\psi_x$ are smooth away from $x$ (hence near $E$, too)  and $\vp \in L^1(X, \om_X^n)$ is smooth near $x$ so that the current  $(\vp \bar \partial \psi_x+\psi_x\partial \vp)\wedge T_k$ is well-defined globally on $X$. 
\end{proof}

 \subsection{Application to singular K\"ahler-Einstein metrics} \label{sec:KE}

  \subsubsection{Intrinsic distance} \label{sec:Loja}

We assume here that $V \subset \PP^N(\C)$ is 
a closed irreducible analytic subvariety.
We let $\Gamma(V)$ denote the set of continuous
arcs $\gamma:[0,1] \rightarrow V$ that are
piecewise ${\mathcal C}^1$-smooth, and set
$$
\ell(\gamma):=\int_0^1 \|\gamma'(t)\|_{\omega_{\rm FS}}dt.
$$

\begin{defi}
For $(x,y) \in V$, we set
$$
d_V(x,y)=\inf \left\{ \ell(\gamma), \; \gamma \in \Gamma(V)
\text{ with } \gamma(0)=x \text{ and } \gamma(1)=y \right\}.
$$
\end{defi}

It is a classical fact that $d_V$ is a distance
such that $d_{\PP^N} \le d_V$. It is perhaps less known that 
one can also bound $d_V$  from above as follows.

\begin{prop} \label{pro:Loja}
There exists $C>0$ and $\alpha \in (0,1]$ such that
$$
d_{\PP^N} \le d_V \le C d_{\PP^N}^{\alpha}.
$$
\end{prop}

We thank Philippe Eyssidieux for providing the references.
Of course one can take $\alpha=1$ if $V$ is smooth, while 
one necessary has $\alpha<1$ already for cusps in dimension $1$.

\begin{proof}
Consider the $U(N+1)$-equivariant real-algebraic embedding
$$
f:L \in \PP^N(\C) \rightarrow P_L \in M_{N+1}(\C),
$$
which sends a complex line $L$ to $P_L$, the orthogonal projection to $L$.
By equivariance we have $f^* ds^2_{M_{N+1}(\C)}=A ds^2_{\rm FS}$, hence
$d_{\PP^N}$ and the pull-back of the euclidean metric $f^*d_{\rm eucl}$
are biLipschitz equivalent. 

Since $f(V)$ is a compact connected semialgebraic subset of $M_{N+1}(\C)$,
the result therefore follows from a result of Lojasiewicz that
we state below.
\end{proof}

\begin{thm} [Lojasiewicz 65]
Let $W \subset \R^N$ be a compact connected semialgebraic set.
There exists $K>0$ and $\alpha>0$ such that for every 
$(x,y) \in W$, one can find a piecewise analytic arc drawn on $W$
and joining $x$ to $y$, whose length is less than $K \|x-y\|^{\alpha}$.
\end{thm}

We refer the reader to \cite{Loj65} for the original statement, and to
\cite{KuO97} for a more recent treatment.

\smallskip

Let $\f:X \rightarrow \R$ be a {\it continuous} quasi-psh function.
It follows from Proposition \ref{pro:Loja} that 
one can measure its modulus of continuity
equivalently  by using $d_V$ or $d_{\PP^n}$.
In particular being H\"older continuous is a notion that is intrinsically
well defined, although the exponent of H\"olderianity is not.

\subsubsection{Singular K\"ahler-Einstein metrics}


In this section, 
 $(V,\om_V)$ is a compact Kähler space, and 
 $V_{\rm reg}$  (resp. $V_{\rm sing}$) denotes its regular (resp. singular) locus.

K\"ahler-Einstein currents
 $\omega_{\rm KE}$ have been constructed on mildly singular compact Kähler spaces 
in \cite{EGZ09,BBEGZ19}. Their two defining features are as follows. 
  
\begin{itemize}
\item They are K\"ahler forms on $V_{\rm reg}$ such that $\Ric(\omega_{\rm KE})$
is proportional to $\omega_{\rm KE}$,
\item They extend to positive currents with bounded local potentials
across  $V_{\rm sing}$.
\end{itemize}
The asymptotic behavior of $\om_{\rm KE}$ near the singularities remains to be understood.

\smallskip

Let us now be a bit more precise and write down the general setting in which the aforementioned objects are well-defined. We assume that $V$ has log terminal singularities; in particular $K_V$ is a $\Q$-Cartier divisor. If $h$ is a smooth hermitian metric on $K_V$ and $\sigma$ is a local generator of $mK_V$, then 
\[\mu_h:=i^{n^2} \frac{(\sigma\wedge \bar \sigma)^{\frac 1m}}{|\sigma|^{2/m}_{h^{\otimes m}}}\]
defines a positive measure on $V_{\rm reg}$ with finite mass (by the log terminal condition), independent of the choice of $m$ or $\sigma$. We extend it to $V$ trivially. Finally, let $\omega_V$ be a Kähler metric on $V$, and consider the Monge-Ampère equation
\begin{equation} 
\label{MAV}
(\omega_V+dd^c \varphi)^n=e^{\lambda \varphi} \mu_h,
\end{equation}
for $\lambda \in \mathbb R$ and $\varphi \in \PSH(V, \omega_V)\cap L^{\infty}(V)$. We scale $h$ so that $\int_V d\mu_h = \int_V \omega_V^n$.

If $c_1(K_V)$ has a sign (i.e. it contains a multiple of a Kähler metric), then one can choose $\lambda,h,\omega_V$ such that $i\Theta_h(K_V)=\lambda \omega_V$. Then, if $\varphi$ is a solution of \eqref{MAV}, the current $\omega_{\varphi}:=\omega_V+dd^c \varphi$ satisfies $\Ric \omega_{\varphi}= -\lambda \omega_{\varphi}$ in the weak sense. 

\begin{defi}
Let $V$ be a compact Kähler space with log terminal singularities such that $c_1(K_V)$ has a sign. A Kähler-Einstein metric is a positive current of the form $\omega_{\rm KE}=\omega_V+dd^c \varphi_{\rm KE}$ where $\varphi_{\rm KE} \in \PSH(V, \omega_V)\cap L^{\infty}(V)$ solves \eqref{MAV} and where $(\lambda, h, \omega_V)$ satisfies $i\Theta_h(K_V)=\lambda \omega_V$.
\end{defi}

If $\lambda \ge 0$, the equation \eqref{MAV} admits a unique (normalized) solution \cite{EGZ09}, but as in the smooth case the situation in the case $\lambda <0$ is more complicated \cite{BBEGZ19}. In particular, if $K_V$ is ample of numerically trivial, then $V$ admits a Kähler-Einstein metric. Moreover, it is showed in {\it loc. cit.} that any Kähler-Einstein metric $\om_{\rm KE}$ induces a smooth Kähler metric on $V_{\rm reg}$.  \\

%
The resolution of \eqref{MAV} 
goes as follows. Let $\pi: X\to V$ be a resolution of singularities of $V$ and let $dV_X$ be a smooth volume form on $X$. The pull back measure $\pi^*\mu_h$ is well-defined and satisfies $\pi^*\mu_h=f dV_X$ where $f\in L^p(dV_X)$ for some $p>1$. Solving \eqref{MAV} is then equivalent to solving the degenerate complex Monge-Amp\`ere equation
\begin{equation} 
\label{MAf}
(\omega+dd^c \psi)^n=e^{\lambda \psi} f dV_{X},
\end{equation}
where $\omega=\pi^* \omega_V$ is semipositive but degenerate along the exceptional locus of $\pi$. 

\smallskip

An immediate corollary of Lemma~\ref{exp} is the following. 

\begin{lem}
Let $V$ be a compact Kähler space of dimension $n$ with log terminal singularities admitting a Kähler-Einstein metric $\om_{\rm KE}$. Let $d_{\rm KE}$ be the geodesic distance induced by the Kähler metric ${\om_{\rm KE}}|_{V_{\rm reg}}$. There exists $\alpha >0$ such that for any $x\in V_{\rm reg}$, we have
\[\int_{V_{\rm reg}} e^{\alpha d_{\rm KE}(\cdot,x)} \om_{\rm KE}^n<+\infty.\]  
\end{lem}

\begin{proof}
Let $\pi: X\to V$ be a resolution of singularities of $V$, let $E$ be the exceptional divisor of $\pi$ and let $\omega_X$ be a Kähler metric on $X$. The current $T:=\pi^*{\om_{\rm KE}}$ induces a smooth Kähler metric on $X\setminus E$. We can see the function  $\rho_x:= d_{\rm KE}(\cdot,x)$ as a Lipschitz function on $X\setminus E$ which extends to $L^1$ function on $X$ by Lemma~\ref{dist L2}. Moreover, Lemma~\ref{exp} implies that  $e^{\beta \rho_x}\in L^1(X,\om_X^n)$ for some $\beta>0$.

Next, as we explained above, one can write $\pi^*{\om_{\rm KE}}^n=f\om_X^n$ where $f\in L^p(X)$ for some $p>1$. If $q$ is the conjugate exponent of $p$, and $\alpha:=\beta/q$, we have by Hölder inequality
\[\int_X e^{\alpha \rho_x} \pi^*\om_{\rm KE}^n \le (\int_X e^{\beta \rho_x} \om_{X}^n)^{1/q}(\int_X f^p \om_X^n)^{1/p}\]
and the lemma follows.  
\end{proof}

\subsubsection{Finiteness of ${\rm diam}(V_{\rm reg},\omega_{\rm KE})$}


It is known that the local potentials $\f_{\rm KE}$ are often continuous near $V_{\rm sing}$
(see \cite{GGZ23}).
However examples from Section \ref{sec:Riccibound} show that an extra information on the modulus of continuity
is required, in order to get control on the diameter of $(V_{\rm reg},\omega_{\rm KE})$.

When $\omega$ is a K\"ahler form on $X$, it follows from 
\cite{Kol08,DDGHKZ14} that the unique (normalized)
solution to \eqref{MAf} is H\"older continuous.
The K\"ahler-Einstein potentials $\f_{\rm KE}$ are also H\"older continuous if
the singularities  $V_{\rm sing}$ are quotient. If $V_{\rm sing}$ consists of isolated
and globally smoothable ordinary double points, it has been shown  
by Hein-Sun \cite{HS} that $\omega_{\rm KE}$ is asymptotic to 
the Stenzel metric $dd^c |z|^{2\frac{n-1}{n}}$, whose potential is also H\"older continuous
(see \cite{CS22} for some recent generalization).

\smallskip

For all these reasons we conjecture the following result.

\begin{conj} \label{conj:KEcont}
The K\"ahler-Einstein potentials 
$\f_{\rm KE}$ are H\"older continuous on $V$.
\end{conj}

Here, Hölder continuous is meant with respect to the euclidean distance 
 induced by local embeddings $V\hookrightarrow \mathbb C^N$. A straightforward application of Theorem~\ref{thm:diameter general} (or simply its weaker version mentioned a few lines above it) allows one to
establish the finiteness of ${\rm diam}(V_{\rm reg},\omega_{\rm KE})$ assuming Conjecture \ref{conj:KEcont}.

\begin{thm}
 \label{thm:diamKE}
If Conjecture \ref{conj:KEcont} is correct, then 
 ${\rm diam}(V_{\rm reg},\omega_{\rm KE})<+\infty$.
\end{thm}

More generally, we conjecture that any solution of \eqref{MAV} is Hölder continuous. Note that if $\varphi$ solves \eqref{MAV}, then $\omega_\varphi:=\omega_V+dd^c\varphi$ induces a Kähler metric on $V_{\rm reg}$. Similarly to Theorem~\ref{thm:diamKE}, the latter conjecture would imply that the diameter of $(V_{\rm reg}, \omega_{\varphi})$ is finite. 

\begin{proof}
Let $\pi: X\to V$ be a log resolution of $V$ and let $\omega_X$ be a Kähler form on $X$ such that $\omega_X\ge \pi^*\om_V$. Let $T:=\pi^*\omega_{\rm KE}=\pi^*\om_V+dd^c(\pi^*\varphi)$. The closed positive current $T$ is a Kähler form outside of the exceptional locus $E$ of $\pi$.  On $X^2$, we have $d_{\om_X} \ge \pi^*d_{\om_V}$.  Moreover, under Conjecture~\ref{conj:KEcont}, the function $\varphi$ is Hölder continuous with respect to $d_{\om_V}$, hence $\pi^*\varphi$ is Hölder continuous with respect to $d_{\om_X}$. The result now follows from Theorem~\ref{thm:diameter general} applied to the current $T$.  
\end{proof}

\begin{rem}

Conjecture~\ref{conj:KEcont} is very strong and we do not need its full force to derive finiteness of the diameter of $(V_{\rm reg}, \om_{\rm KE})$. Indeed, the proof above shows that it would be enough to prove that for a log resolution $\pi: X\to V$ of $V$, the function $\pi^*\varphi$ has a modulus of continuity $m_\varphi$ satisfying $m_{\varphi}(r) \lesssim (\log \log (-r))^{-1-\delta}$ for some $\delta>0$. 
\end{rem}

\section{Examples} \label{sec:Riccibound}

\subsection{Radial Examples} \label{sec:radialexamples}

   We assume here that the quasi-psh functions $\f$ under consideration are smooth 
   in $X \setminus \{ p \}$. We choose a local chart biholomorphic to the unit ball $B$ of $\C^n$,
   with $p$ corresponding to the origin. We further assume that $\f$
   has a {\it radial singularity} at $p$, i.e. it is invariant under the group  
  $U(n,\mathbb C)$ near $p$, and so is  $\omega$.
   The  
   singularity type  of   $\f$ 
    only depends on its local behavior near $p$.
   
 We thus consider the local situation of 
  a psh function $v=\chi \circ L$, 
  where $\chi:\R^- \rightarrow\R^-$ is a 
 smooth  strictly convex increasing function and $L:z \in  B \rightarrow \log |z|^2 \in \R^-$.
  A standard computation shows that
\[ \omega:=dd^c v=(\chi' \circ L) \, dd^c L +(\chi'' \circ L) \, dL \wedge d^c L\]
  and
$$
  (dd^c v)^n =c_n \frac{(\chi' \circ L)^{n-1}\chi'' \circ L}{|z|^{2n}} dV_{\rm eucl}
=c_n \big[\chi'^{n-1}\chi'' e^{-n \cdot})\circ L\big] \, dV_{\rm eucl}
 $$

 
 \subsubsection*{Diameter}
 One can rewrite the Kähler form $\omega$ using the usual coordinates on $\mathbb C^n$ as
\begin{eqnarray*}
\om &=& \chi'\cdot  \frac {\sqrt{-1}}{|z|^2} \sum_{i,j} \Big( \delta_{ij}- \frac{\bar z_i z_j}{|z|^2}\Big) \, dz_i\wedge d\bar z_j+\chi''\cdot  \frac {\sqrt{-1}}{|z|^2} \sum_{i,j}  \frac{\bar z_i z_j}{|z|^2} \, dz_i\wedge d\bar z_j\\
&=&\omega_s+\omega_r
\end{eqnarray*}
where the composition with $L$ has been omitted. The subscripts $s$ (resp. $r$) stand for spherical (resp. radial). Indeed, the kernel of $\omega_s$ in $T_{\mathbb C^n}^{1,0}$ is generated by the holomorphic radial vector field $\xi:=\sum_{i=1}^n z_i \frac{\partial}{\partial z_i}$ while $\omega_r$ has rank one and $\|\xi\|^2_{\omega_r}=1$. \\

We claim that the rays through the origin of $\mathbb C^n$ are geodesics for $\omega$. Indeed, if $x\in \mathbb C^n$ say of norm one and $\gamma_r(t)=tx$ for $t\in [\ep,1]$, then we have $\gamma'_r(t)= t^{-1}(\xi+\bar \xi)_{\gamma_r(t)}$ hence $\|\gamma_r'(t)\|^2_{\omega}=t^{-2}\chi''(L(\gamma_r(t)))$ and the length of $\gamma_r$ with respect to $\omega$ is 
\[\ell(\gamma_r)=\int_{\log \ep}^{0} \sqrt{\chi''(s)} ds.\]
 Now, let $\gamma$ be any path connecting $\ep x$ and $x$. Since $dL\otimes d^cL= (d \log r)^2$, we have
 \begin{eqnarray*}
\ell(\gamma) &=& \int_{\ep}^1 \|\gamma'(t)\|_{\omega} \,dt 
\ge  \int_{\ep}^1 \|\gamma'(t)\|_{\omega_r} \,dt\\
&= & \int_{\ep}^1\sqrt{\chi''(L(\gamma(t))}\frac{|dr(\gamma'(t))|}{r(\gamma(t))}\, dt
\ge  \int_{\log \ep}^{0} \sqrt{\chi''(s)} ds
\end{eqnarray*}
where in the last line we have used $|dr|\ge dr$ and performed the change of variable $s:=r(\gamma(t))$. In particular, we find that $\ell(\gamma)\ge \ell(\gamma_r)$ which proves that $\gamma_r$ is a geodesic. Tracing back the inequalities, we actually see that any geodesic connecting $\ep x$ and $x$ must have $\|\gamma'(t)\|_{\omega_s}=0$, hence it must coincide with $\gamma_r$.  \\

Next, if $x,y\in B^*$, we can connect $x$ to $y$ using a radial ray, a spherical geodesic on $\mathbb S^{2n-1}$ and another radial ray.  Since $\omega$ is smooth away from the origin, this shows that  the Riemannian metric associated to $\omega$ on $B^*$ has finite  diameter if and only if
\[{\rm diam}(B^*,d_\omega)<+\infty
\; \; \Longleftrightarrow \; \; \int_{-\infty} \sqrt{\chi''(s)} ds <+\infty.\]

\subsubsection*{Modulus of continuity}
 
We restrict ourselves to continuous potentials, i.e. we assume that $\chi(-\infty)>-\infty$.
We can normalize $v$ without loss of generality so that $v(0)=\chi(-\infty)=0$. 
Since $v$ is smooth outside of the origin,
its modulus of continuity is  
$$
m_v(r):=\sup_{t <\log r} \chi(t).
$$

 \subsubsection*{Integrability condition (K)}
 
 The Monge-Amp\`ere measure  $(dd^c v)^n=f \, dV$ satisfies
 \CK if there exists 
 an increasing function $h$ such that $\int^{+\infty} h(t)^{-1} dt <+\infty$
 and
 $$
  \int_{B^*} f ( \log f)^n ( h \circ \log \circ \log f)^{n} dV <+\infty.
 $$
 In order to simplify the next computation, we will make the assumption that $\chi(t)$ does not go to zero too fast at infinity. More precisely, we impose that $\chi'(t),\chi''(t) \ge e^{Ct}$ near $t=-\infty$ for some $C>0$. This guarantees that $\log f \sim \log |z|$, so that $v$ satisfies \CK iff
   $$
  \int_{-\infty} \chi''(t) (\chi'(t))^{n-1}  |t|^n (h \circ \log |t|)^{n} dt <+\infty.
  $$

   \subsubsection*{Ricci lower bound}
   
   As the computations are   involved,
   we analyze whether the Ricci curvature of the above radial metrics is bounded below
   in a separate section \ref{sec:ricci}. We will actually be interested in the refined question of whether the  approximants $\ome=:dd^c \chi \circ \log (\zze)$ have a lower bound on their Ricci curvature independent on $\ep$.  

  \subsection{Explicit examples} \label{exa:radial}
In what follows, we give four families of examples and check whether the various conditions mentioned above are satisfied.  

\medskip  
  
 \fbox{ {\bf 1.}} Consider $\chi_{\a}(t)=\exp(\a t)$ for some $\a>0$.
  
  \medskip
  
  \noindent
  We have 
  $$
  {\rm Ric}(\omega_{\a})=-dd^c \log f=n(1-\a)dd^c \log |z|^2.
  $$
  In particular, $\Ric \, \om_{\a} \ge 0$ if and only if $\a \le 1$. Moreover, since $\chi'$ tends to zero at $-\infty$, one can easily see that the Ricci curvature of $\omega_{\a}$ is unbounded below when $\alpha>1$. 
  Thus
  \begin{itemize}
  \item[$\bullet$] \CK is satisfied;
   $\int_X f_{\a} |\log f_{\a}| \log ( \log (f_{\a}+3))^{p} dV <+\infty$ for all $p$;
  \item[$\bullet$] $m_{v_{\a}}(r) \sim r^{\a}$ and $\int_0 t^{-1}\sqrt{m_{v_{\a}}(t)} dt <+\infty$;
  \item[$\bullet$] ${\rm diam}(B^*,d_{\omega_{\a}})<+\infty$ for all $\a>0$;
  \item[$\bullet$] $\Ric \, \om_{\a} \ge -A \omega_{\a}$ if and only if $0<\a \le 1$.
  \end{itemize}
  
  \medskip
  
 \fbox{ {\bf 2.}} Consider $\chi_{\a}(t)=(-t)^{-\a}$ for some $\a>0$.
  
    \medskip
  
  \noindent
   We have
  $m_{v_{\a}}(r)=(-\log r)^{-\a}$, hence
  $
  \int_0 \frac{\sqrt{m_v(r)}}{r}dr= \int_0 \frac{dr}{r(-\log r)^{\a/2}}
  $
  is finite if and only if $\a>2$.
    One has 
$$\chi^{(k)}(t)=\prod_{j=0}^{k-1}(\alpha+j) \cdot (-t)^{-\alpha-k}$$
so that with the notation 
from Section \ref{sec:ricci}, 
$\lambda = n-\frac{n\alpha+n+1}{(-\log \|z\|^2)} \ge 0$ 
for $|z|\ll1$ 
but 
$$\frac{\mu}{\chi''} = -\frac{n(\alpha+1)+1}{\alpha(\alpha+1)} \cdot (-\log  \zz)^{\alpha} $$
which goes to $-\infty$ when $z\to 0$, hence
$\Ric \omega_{\a}$ is unbounded from below.
 Thus
  \begin{itemize}
  \item[$\bullet$] \CK is satisfied;
  $\int_X f_{\a} |\log f_{\a}| \log ( \log (f_{\a}+3))^{p} dV <+\infty$ for all $p$;
 \item[$\bullet$] $\int_0 t^{-1} \sqrt{m_{v_{\a}}(t)} dt <+\infty$ if and only if $\a>2$;
  \item[$\bullet$] ${\rm diam}(B^*,d_{\omega_{\a}})<+\infty$ for all $\a>0$;
  \item[$\bullet$] $\Ric \, \om_{\a}$ is unbounded from below.
  \end{itemize}
  
  \medskip
  
 \fbox{ {\bf 3.}}  Consider $\chi_{\a}(t)=(\log (-t))^{-\a}$, where $\a>0$.
 
   \medskip
  
  \noindent
  One computes
  $
  \chi_{\a}'(t)=\frac{\a}{(-t) (\log (-t))^{1+\a}}$
and  $ \chi_{\a}''(t) \sim \frac{\a}{t^2 (\log (-t))^{1+\a}},  $
  so 
  $$
  {\rm diam}(B^*,d_{\a})<+\infty
\Longleftrightarrow 
\int_{-\infty} \frac{dt}{|t| (\log (-t))^{\frac{1+\a}{2}}} dt<+\infty
\Longleftrightarrow
\a>1.
  $$
  
 We obtain  $m_{v_{\a}}(r)=(\log (-\log r))^{-\a}$ 
 hence $\int_0 \frac{\sqrt{m_{v_{\a}}(t)}}{t}dt =+\infty$. 
 Observe that
 $$
  \int_{-\infty} \chi_{\a}''(t) (\chi_{\a}'(t))^{n-1}  |t|^nh(\log |t|)^n dt 
\sim  \int_{-\infty} \frac{h(\log |t|)^n dt}{|t| (\log |t| )^{n(1+\a)}}=\int^{+\infty} \frac{h(s)^nds}{s^{n+n\alpha}}.
  $$
If $n\a>1$, then this integral converges for $h(s)=s^{1+\ep}$ as long as $0< \e\ll 1$. Moreover, Hölder inequality with $p=\frac{n+1}n$ and  $q=n+1$ yields:

\[+\infty=\int^{\infty} \frac {ds} s= \int^{+\infty} h(s)^{-\frac{n}{n+1}} \frac{h(s)^\frac{n}{n+1} }{s}ds  \le \left(\int^{+\infty}\frac {ds}{h(s)}\right)^{\frac{n}{n+1}}\left(\int^{+\infty}\frac{h(s)^nds}{s^{n+1}}\right)^{\frac 1{n+1}}\]
hence $\int^{+\infty} \frac{h(s)^nds}{s^{n+n\alpha}}$ is divergent as soon as $n\alpha \le 1$:
\CK is satisfied iff $\a>1/n$.

A computation shows  that for every $k\ge 1$, one has 
$$\chi_{\a}^{(k)}(t)=\frac{(k-1)! \cdot \alpha}{(-t)^k(\log(-t))^{\alpha+1}}
\Big( 1+O\Big(\frac 1{\log(-t)}\Big)\Big)$$
when $t\to -\infty$. 
The eigenvalue $\mu$ of $\Ric \om_{\a}$ satisfies 
$$\frac{\mu}{\chi_{\a}''} \underset{\|z\| \to 0}{\sim} -\frac{n+1}{\alpha} \cdot (\log(-\log  \zz))^{\alpha+1} $$
which is unbounded below. In particular $\Ric \omega_{\a}$ is unbounded from below. 
Thus
  \begin{itemize}
  \item[$\bullet$] \CK is satisfied if and only if $\a>1/n$; 
  \item[$\bullet$] $\int_X f_{\a} |\log f_{\a}| \log ( \log (f_{\a}+3))^{p} dV <+\infty$ 
  iff $p < n(1+\alpha)-1$;
   \item[$\bullet$] $\int_0 t^{-1} \sqrt{m_{v_{\a}}(t)} dt =+\infty$ for all $\a>0$;
  \item[$\bullet$] ${\rm diam}(B^*,d_{\omega_{\a}})<+\infty$ if and only if $\a>1$;
  \item[$\bullet$] $\Ric \, \om_{\a}$ is unbounded from below.
  \end{itemize}
  In particular 
  we obtain 
   $\int_X f_{1} |\log f_{1}| \log ( \log (f_{1}+3))^{p} dV <+\infty$ for all $p<2n-1$,
  while ${\rm diam}(B^*,d_{\omega_{1}})=+\infty$
  and \CK is satisfied (if $n \geq 2$).

  \medskip
 
   \fbox{ {\bf 4.}}  Consider $\chi_{\a}(t)=-(\log (-t))^{\a}$, where $\a>0$.

    \medskip
  
  \noindent
  Of course in that case the potential is unbounded but there is interesting geometric behavior associated to these metrics. 
We have 
\[\chi_{\a}'(t)=\frac{\a}{(-t) (\log (-t))^{1-\a}}, \quad \mbox{and} \quad   \chi_{\a}''(t) \sim \frac{\a}{t^2 (\log (-t))^{1-\a}}.\] Since $\int_{-\infty} \frac{dt}{|t| (\log (-t))^{\frac{1+\a}{2}}} dt=+\infty$, we infer that  ${\rm diam}(B^*,d_{\a})=+\infty.$
 Next, we have 
 $$
  \int_{-\infty} \chi_{\a}''(t) (\chi_{\a}'(t))^{n-1}  |t|^n h(\log |t|)^n dt 
=  \int_{-\infty} \frac{h(\log |t|)^n(\log |t| )^{n\a}dt}{|t| \log^n |t|}=\int^{+\infty}\frac{h(s)^nds}{s^{n(1-\alpha)}}=+\infty
  $$
  since $h\ge 1$ near $+\infty$. Thus \CK is never satisfied, as expected. However, the density $f$ of $\omega_\alpha$ with respect to a smooth volume form $dV_X$ satisfies
   $$
  \int_X f (\log f)^n (\log \log (f+3))^p dV_X \sim \int^{+\infty} \frac{ds}{s^{n-n\alpha-p}}
  $$
  which converges iff $p<n-1-n\alpha$. 
  Taking $0<\alpha \ll 1$, we get a density 
   such that $f( \log f)^n (\log \log (f+3))^{n-1-\ep}\in L^1$, which shows how close to sharp \CK is.

\medskip

As for the Ricci curvature, we will only treat the case $n=\alpha=1$ for simplicity, since we already see an interesting phenomenon appear. More precisely, if $n=\alpha=1$, then we are just considering the Poincaré metric on the punctured disk. 
Thus
  \begin{itemize}
  \item[$\bullet$] \CK is never satisfied; $v_{\alpha}$ is unbounded;
  \item[$\bullet$] ${\rm diam}(B^*,d_{\omega_{\a}})=+\infty$ for all $\a>0$;
  \item[$\bullet$] $\Ric \, \om_{\a}=-\omega_{\alpha}$ if $\alpha=n=1$.
  \end{itemize}
  
The Ricci curvature of the smoothing $\widetilde{\om}_\ep:=dd^c (-\log (-\log(|z|^2+\ep^2)))$,  is 
however unbounded from below as $\ep \to 0$. Indeed 
 set $\chi(t)=-\log(-t), F:=\ep^2e^{-t}\in (0,1]$ and $s:=-t\ge 0$. Here, $t=\log(|z|^2+\ep^2)$. Next, set $\psi:=\chi''+(\chi'-\chi'')F$, so that 
   \[\psi= \frac 1{s^2}+F\Big(\frac 1s-\frac {1}{s^2}\Big), \quad \psi'=\frac 2{s^3}+F\Big(-\frac 1s+\frac 2{s^2}-\frac{2}{s^3}\Big),\quad  \psi''=\frac 6{s^4}+F\Big(\frac 1s-\frac 3{s^2}+\frac{6}{s^3}-\frac{6}{s^4}\Big).\]  
   From the computations in \textsection~\ref{sec:ricci}, the Ricci curvature of 
   $\widetilde{\om}_\ep$ is bounded from below uniformly if and only if we have $C>0$ such that 
   \begin{equation}
   \label{minoration}
   F\psi^2+(1-F)\psi'^2-F\psi\psi'-(1-F)\psi\psi'' \ge -C \psi^3.
   \end{equation}
   
   {\it Case 1.} When $F$ stays bounded away from $0$ (i.e. very close to the origin). Then we have the asymptotics
 $
 \psi \approx \frac F s, \quad\psi'\approx -\frac F {s}, \quad\psi'' \approx \frac F{s}.
 $
The LHS of \eqref{minoration} is equivalent to $\frac{2F^3}{s^2}$ hence the inequality is satisfied for large $C$.\\

{\it Case 2.} When $F$ get close to zero (relatively far away from the origin). Brute force computations show that the asymptotics of the LHS are \[-\frac 2 {s^6}-\frac{F}{s^3}+\frac{2F^3}{s^2}.\]
Since $\psi^3 \approx \frac{1}{s^6}+\frac{F^3}{s^3}$, one has $\frac{F}{s^3} \gg \frac{F^3}{s^3}$. If we work near $|z|^2=\ep^2(-\log \ep)^\alpha$ with $\alpha\in (\frac 12, 3)$, then one can check that  $\frac{F}{s^3} \gg \max\{\frac 1 {s^6}, \frac{F^3}{s^2}\}$. This shows that \eqref{minoration} is violated.

\subsection{Lower bounds on the Ricci curvature} \label{sec:ricci}

  We analyze whether the Ricci curvature of smooth K\"ahler metrics approximating the radial examples is bounded below.

\subsubsection*{Computation of the metric}
We set 
$\ome:=dd^c \Big[\chi \circ \log(\zze)\Big]$
and observe that 
\begin{equation}
\label{ddc}
\ome = \frac {\sqrt{-1}}{\zze} \sum_{i,j} \Big(\chi' \delta_{ij}+(\chi''-\chi') \cdot \frac{\bar z_i z_j}{\zze}\Big) \, dz_i\wedge d\bar z_j
\end{equation}
where we write $\chi'$ in place of $\chi' \circ \log(\zze)$ and similarly for $\chi''$. In view of the expression above, it is convenient to introduce the matrix 
$$A(z):=\left( \frac{\bar z_i z_j}{\zze}\right)_{ij}$$
It is semipositive hermitian with rank one and its non-zero eigenvalue coincides with its trace, i.e. $\frac{\zz}{\zze}$. With these notations, the Kähler form $\ome$ is associated to the hermitian matrix 
$$H(\ome):=\frac1\zze \left(\chi' I+(\chi''-\chi') A(z)\right).$$
Its eigenvalues are 
$\frac {\chi'}{\zze}$
with multiplicity $n-1$ and $$\frac 1 \zze \left(\chi'+(\chi''-\chi') \cdot  \frac{\zz}{\zze}\right)=\frac 1 \zze \left(\chi''+(\chi'-\chi'') \cdot  \frac{\ep^2}{\zze}\right)$$ with multiplicity $1$. \\

A useful computation  is that of $dd^c \log f $, where $f$ is positive non-decreasing and convex and as usual, we omit the composition with $\log(\zze)$. Then 
\begin{align*}
dd^c \log f &= \frac{dd^c f}{f}-\frac {df\wedge d^c f}{f^2}\\
&=\frac{\sqrt{-1}}{\zze}\sum_{i,j} \Big(\frac {f'}f \delta_{ij}+\Big[\frac{f''-f'}{f}-\frac{f'^2}{f^2}\Big]\cdot \frac{\bar z_i z_j}{\zze}\Big) \, dz_i\wedge d\bar z_j
\end{align*} 
so that 
$$H(dd^c \log f)=\frac{1}{(\zze)\cdot f^2} \left([f\cdot f']\cdot I +[f(f''-f') -f'^2]\cdot A(z)\right).$$

\subsubsection*{Computation of the Ricci curvature}
The determinant of $H(\ome)$ is given by
$$
\det H(\ome)=\frac 1{(\zze)^n} \cdot \chi'^{n-1} \cdot \left(\chi''+(\chi'-\chi'') \cdot  \frac{\ep^2}{\zze}\right)
$$
and therefore, the Ricci curvature of $\ome$ is given by 
$$\Ric \ome = \underbrace{n\,  dd^c \log(\zze)}_{=:{\bf \rm (I)}} \underbrace{-(n-1) dd^c \log \chi'}_{=:{\bf \rm (II)}} \underbrace{- dd^c \log  \left(\chi''+(\chi'-\chi'') \cdot  \frac{\ep^2}{\zze}\right)}_{=:{\bf \rm (III)}}$$
and we can decompose the hermitian matrix $H(\Ric \ome)$ associated to $\Ric \ome$ as $$H(\Ric \ome)=\frac 1 \zze (H_1+H_2+H_3)$$ 
where
$H_1= nI-nA(z)$
has eigenvalues
$$\begin{cases} 
\lambda_1 =n \quad \mbox{with multiplicity } (n-1)\\
\mu_1=\frac{n \ep^2}{\zze} \quad \mbox{with multiplicity } 1
\end{cases}$$
and
\begin{align*}
H_2&=-(n-1) (\zze ) H(dd^c \log \chi')\\
&=\frac{-(n-1) }{\chi'^2}\left( [\chi'\cdot\chi''] \cdot I +\Big[\chi'(\chi'''-\chi'')-\chi''^2\Big] \cdot A(z)\right).
\end{align*}
has eigenvalues
$$\begin{cases} 
\lambda_2 =-(n-1)\frac{ \chi''}{ \chi'} \quad \mbox{with multiplicity } (n-1)\\
\mu_2=-(n-1)\Big[ \Big(\frac{\chi'''}{\chi'}-\big(\frac{\chi''}{\chi'}\big)^2\Big) - \frac{\ep^2}{\zze} \Big(\frac{\chi'''-\chi''}{\chi'}-\big(\frac{\chi''}{\chi'}\big)^2\Big)\Big]\quad \mbox{with multiplicity } 1
\end{cases}$$
As for $H_3$, it is convenient to set 
$$
\psi(t):=  \chi''(t)+(\chi'(t)-\chi''(t))\cdot \ep^2e^{-t}\\
=\chi''(t)-\ep^2 (\chi'(t)e^{-t})'
$$
so that 
$$
H_3=-(\zze) H(dd^c \log \psi)
=-\frac{1}{\psi^2}\left( [\psi\cdot\psi'] \cdot I +\Big[\psi(\psi''-\psi')-\psi'^2\Big] \cdot A(z)\right).
$$
has eigenvalues
$$\begin{cases} 
\lambda_3 =\frac{\psi'}{ \psi} \quad \mbox{with multiplicity } (n-1)\\
\mu_3=- \Big[ \Big(\frac{\psi''}{\psi}-\big(\frac{\psi'}{\psi}\big)^2\Big) - \frac{\ep^2}{\zze} \Big(\frac{\psi''-\psi'}{\psi}-\big(\frac{\psi'}{\psi}\big)^2\Big)\Big]\quad \mbox{with multiplicity } 1
\end{cases}$$
Therefore, the eigenvalues $\lambda:=\sum_{i=1}^3 \lambda_i$ and $\mu:=\sum_{i=1}^3 \mu_i$ with respective multiplicity $(n-1)$ and $1$ of $H_1+H_2+H_3$ are
$\lambda=n-\Big[(n-1) \frac {\chi''}{\chi'}+\frac{\psi'}{\psi}\Big]$
and
\begin{align*}
\mu=&- \Big[(n-1) \Big(\frac{\chi'''}{\chi'}-\big(\frac{\chi''}{\chi'}\big)^2\Big)+ \Big(\frac{\psi''}{\psi}-\big(\frac{\psi'}{\psi}\big)^2\Big) \Big] \\
&+\frac{\ep^2}{\zze} \Big[n +\Big\{(n-1)\Big(\frac{\chi'''-\chi''}{\chi'}-\big(\frac{\chi''}{\chi'}\big)^2\Big)+ \Big(\frac{\psi''-\psi'}{\psi}-\big(\frac{\psi'}{\psi}\big)^2\Big)\Big\} \Big]
\end{align*}
and we have $\Ric \ome \ge -C \ome$ for some constant $C>0$ if and only if 
$$
\lambda \ge -C \chi' 
\; \; \text{ and } \; \; 
\mu \ge -C (\chi''+(\chi'-\chi'') \cdot  \frac{\ep^2}{\zze}).
$$

 \subsubsection*{Special case when $\ep=0$}
 In order to show that the Ricci curvature of $\ome$ cannot be uniformly bounded from below, it is enough (though not sufficient in general) to show that the Ricci curvature of $\om$ is unbounded from below. The computations become simpler since one has then 
$
\lambda=n-\Big[(n-1) \frac {\chi''}{\chi'}+\frac{\chi^{(3)}}{\chi''}\Big]
$
and
$$
\mu=- \Big[(n-1) \Big(\frac{\chi'''}{\chi'}-\big(\frac{\chi''}{\chi'}\big)^2\Big)+ \Big(\frac{\chi^{(4)}}{\chi''}-\Big(\frac{\chi^{(3)}}{\chi''}\Big)^2\Big) \Big] 
$$
 and one only has to disprove one of the following inequalities
 $$ 
\lambda \ge -C \chi'
\; \; \text{ or } \; \; 
\mu \ge -C \chi''.
$$

\bibliographystyle{smfalpha}
\bibliography{biblioFKR}

\end{document}